\documentclass{article}

\usepackage{authblk}

\usepackage[utf8]{inputenc}
\usepackage[T1]{fontenc}
\usepackage{textcomp}
\usepackage{amsmath,amssymb,amsthm}
\usepackage{lmodern}
\usepackage[a4paper]{geometry}
\usepackage{xcolor, pict2e}
\usepackage{microtype}
\usepackage{caption}
\usepackage{listings}
\usepackage{multicol}
\usepackage{moreverb}
\usepackage{hyperref}
\hypersetup{pdfstartview=XYZ}
\usepackage{wrapfig}
\usepackage[sans]{dsfont}

\usepackage{hyperref}
\hypersetup{
    colorlinks=true,
    linkcolor=blue,
    citecolor=magenta,
    urlcolor=blue,
    pdfborder={0 0 0}
}

\usepackage[font=sf, labelfont={sf,bf}, margin=1cm]{caption}

\usepackage{float}
\usepackage[caption = false]{subfig}
\usepackage{graphicx}

\newtheorem{theorem}{Theorem}[section]
\newtheorem{proposition}[theorem]{Proposition}
\newtheorem{lemma}[theorem]{Lemma}
\newtheorem{corollary}[theorem]{Corollary}
\theoremstyle{remark}
\newtheorem{remark}[theorem]{Remark}

\newtheorem{lemmaa}{Lemma}

\newcommand{\R}{\mathbb{R}}
\newcommand{\N}{\mathbb{N}}
\newcommand{\Z}{\mathbb{Z}}
\newcommand{\Q}{\mathbb{Q}}
\newcommand{\C}{\mathbb{C}}

\newcommand{\D}{\mathbb{D}}

\newcommand{\Bc}{\mathcal{B}}

\newcommand{\Dc}{\mathcal{D}}

\newcommand{\Fc}{\mathcal{F}}

\newcommand{\Hc}{\mathcal{H}}

\newcommand{\Mc}{\mathcal{M}}

\newcommand{\Sc}{\mathcal{S}}
\newcommand{\Tc}{\mathcal{T}}

\newcommand{\Xc}{\mathcal{X}}
\newcommand{\Zc}{\mathcal{Z}}
\newcommand{\Pc}{\mathcal{P}}

\newcommand{\expect}{\mathbb{E}}
\newcommand{\Expect}[1]{\mathbb{E} \left[ #1 \right] }
\newcommand{\EXPECT}[2]{\mathbb{E}_{#1} \left[ #2 \right] }
\newcommand{\prob}{\mathbb{P}}
\newcommand{\Prob}[1]{\mathbb{P} \left( #1 \right) }
\newcommand{\PROB}[2]{\mathbb{P}_{#1} \left( #2 \right) }

\renewcommand{\P}{\mathbb{P}}

\newcommand{\abs}[1]{\left\vert #1 \right\vert}
\newcommand{\scalar}[1]{\left\langle #1 \right\rangle }
\newcommand{\floor}[1]{\left\lfloor #1 \right\rfloor}
\newcommand{\indic}[1]{ \mathbf{1}_{ \left\{ #1 \right\} } }
\newcommand{\eps}{\varepsilon}

\DeclareMathOperator{\CR}{CR}

\makeatletter
\newcommand*\bigcdot{\mathpalette\bigcdot@{.5}}
\newcommand*\bigcdot@[2]{\mathbin{\vcenter{\hbox{\scalebox{#2}{$\m@th#1\bullet$}}}}}
\makeatother

\usepackage{import}
\usepackage{xifthen}
\usepackage{pdfpages}
\usepackage{transparent}
\usepackage{chngcntr}
\counterwithin{figure}{section}

\usepackage{enumitem}
\newcommand{\subscript}[2]{$#1 _ #2$}

\title{Characterisation of planar Brownian multiplicative chaos}
\author{Antoine Jego}
\date {}

\numberwithin{equation}{section}

\begin{document}

\maketitle

\begin{abstract}
We characterise the multiplicative chaos measure $\Mc$ associated to planar Brownian motion introduced in \cite{bass1994,AidekonHuShi2018,jegoGMC} by showing that it is the only random Borel measure satisfying a list of natural properties. These properties only serve to fix the average value of the measure and to express a spatial Markov property.
As a consequence of our characterisation, we establish the scaling limit of the set of thick points of planar simple random walk, stopped at the first exit time of a domain, by showing the weak convergence towards $\Mc$ of the point measure associated to the thick points. In particular, we obtain the convergence of the appropriately normalised number of thick points of random walk to a nondegenerate random variable. The normalising constant is different from that of the Gaussian free field, as conjectured in \cite{jego2020}. These results cover the entire subcritical regime.

A key new idea for this characterisation is to introduce measures describing the intersection between different independent Brownian trajectories and how they interact to create thick points.
%
%
%
\end{abstract}

\tableofcontents

\section{Introduction and main results}

The study of exceptional points of planar random walk has a long history.
In 1960, Erd\H{o}s and Taylor \cite{erdos_taylor1960} showed that the number of visits of the most visited site of a planar simple random walk after $n$ steps is asymptotically between $(\log n)^2/(4\pi)$ and $(\log n)^2/\pi$ and conjectured that the upper bound is sharp. This conjecture was proven forty years later by Dembo, Peres, Rosen and Zeitouni in the landmark paper \cite{dembo2001}. 
These authors also considered the set of thick points of the walk, where the walk has spent a time at least a fraction of $(\log n)^2$, and computed its asymptotic size at the level of exponents.
Their proof is based on planar Brownian motion and uses KMT-type approximations to transfer the results to random walk with increments having finite moments of all order. \cite{rosen2006} provided another proof of these results without the use of Brownian motion and \cite{BassRosen2007} extended them to planar random walk with increments having finite moment of order $3+\eps$. \cite{jego2020} streamlined the arguments by exploiting the links between the local times and the Gaussian free field (GFF) and extended the above results to walks with increments of finite variance and to more general graphs. \cite{AidekonHuShi2018} and \cite{jegoGMC} constructed simultaneously a random measure supported on the set of thick points of Brownian motion extending results of \cite{bass1994}. Finally, \cite{Okada2016} studied the most visited points of the inner boundary of the random walk range.

A closely related (but in fact distinct as we will argue below) area of research is the study of planar random walk run until a time close to the cover time. It has become very active since Dembo, Peres, Rosen and Zeitouni \cite{DPRZ2004} found the leading order term of the cover time for both planar Brownian motion and random walk settling a conjecture of Aldous \cite{Aldous1989}. Since then, the understanding of the behaviour of the walk in this regime has considerably improved. We mention a few works. On the torus, the multifractal structure of the set of thin/thick/late points has been studied \cite{DPRZ2006,CometsPopov2016,Abe15}, the subleading order of the cover time has been established \cite{Abe2017,BeliusKistler2017} and even the tightness of the cover time associated to Brownian motion on the 2D sphere is known \cite{BeliusRosenZeitouni2019b}. For a walk resampled every time it hits the boundary of a planar domain, the scaling limit of the set of thin/thick/late points has been established \cite{AbeBiskup}. The picture is even more complete on binary trees where the scaling limit of the cover time \cite{CortinesLouidorSaglietti2018,DRZ2019} as well as the scaling limit of the set of extreme points having maximal local times \cite{Abe2018} have been derived.

The current paper is closer to the setup of the first series of articles where the walk is stopped at the first exit time of a planar domain.
Its aim is to establish the scaling limit of the thick points of planar simple random walk stopped at the first exit time of a domain by showing that the point measure associated to the thick points converges to a nondegenerate random measure $\Mc$. This gives much finer information on the set of thick points and, as a corollary, we obtain the convergence of the appropriately normalised number of thick points of random walk to a nondegenerate random variable considerably improving the previously known above-mentioned results. In that sense, it is the final answer to the question raised by Erd\H{o}s and Taylor.

In this regime a comparison to the GFF is too rough, in contrast with the regime corresponding to times closer to cover time; and indeed, in this latter case the limiting measure is related to the so-called Liouville measure of GFF (see \cite{AbeBiskup} and see \cite{RobertVargas2010, DuplantierSheffieldGMC, RhodesVargasGMC, ShamovGMC, berestycki2017} for subcritical Liouville measures and Gaussian multiplicative measures).
In our delicate setting of limited time horizon, the limiting measure $\Mc$, that we can call ``Brownian multiplicative chaos'' in analogy to Gaussian multiplicative chaos measures, was introduced in \cite{bass1994,AidekonHuShi2018,jegoGMC} and was so far fairly mysterious. On the one hand, it shares a lot of similarities with the Liouville measure such as carrying dimension and conformal invariance. But on the other hand the measure $\Mc$ is very different in the sense that it is carried and entirely determined by the random fractal composed of a Brownian trace. One of the main result of this paper consists in characterising the law of the measure $\Mc$. We show that it is the only random Borel measure satisfying a list of natural properties which fix its average value and express a spatial Markov property. This demystifies the measure $\Mc$ and shows its universal nature.

We start by presenting our results on random walk. We then discuss our characterisation of Brownian multiplicative chaos.

\medskip

In this paper, we will consider simply connected domains with a boundary composed of a finite number of analytic curves. Such a continuous domain will be called a ``nice domain'' and a boundary point where the boundary is locally analytic will be called a ``nice point''.

\subsection{Scaling limit of thick points of planar random walk}
\label{subsec:RW}

We will extend the definition of the integer part function by setting for $x=(x_1,x_2) \in \R^2$, $\floor{x} = (\floor{x_1},\floor{x_2})$. For a nice domain $U$, a reference point $x_0 \in U$ and a large integer $N$, let $U_N$ and $\partial U_N$ be discrete approximations of $U$ and $\partial U$ defined as follows:
\[
U_N := \left\{ \floor{Nx}: x \in U,
\begin{array}{c}
\text{there~exists~a~path~in~} \Z^2 \text{~from~} \floor{Nx} \text{~to~} \floor{Nx_0} \\
\text{whose~distance~to~the~boundary~of~} NU \text{~is~at~least~} 1
\end{array}
\right\}
\]
and
\[
\partial U_N := \left\{ x \in U_N: \exists y \in \Z^2 \backslash U_N, \abs{x-y} = 1 \right\}.
\]
This intricate definition of $U_N$ is just to avoid issues with ``thin'' boundary pieces.\footnote{For instance, if the domain $U$ is a disc minus a slit, it might happen that two neighbouring vertices of $U \cap \frac1N \Z^2$ are actually at a macroscopic distance to one another in the internal metric of $U$.}
For $z \in \partial U$, we will abusively write $\floor{Nz}$ any point of $\partial U_N$ closest to $z$.
Let $(X_t)_{t \geq 0}$ be a continuous time simple random walk on $\Z^2$ with jump rate one (at every vertex, it waits an exponential time with parameter one before jumping) and define its hitting time of $\partial U_N$ and local times:
\[
\tau_{\partial U_N} := \inf \left\{ t \geq 0, X_t \in \partial U_N \right\}
\mathrm{~and~for~} x \in \Z^2, t >0,
\ell_x^t := \int_0^t \indic{X_s = x} ds.
\]
For $x, z \in \C$, we will denote by $\prob^{U_N}_x$ the probability measure associated to the walk $(X_t,t \leq \tau_{\partial U_N})$ starting at $X_0 = \floor{x}$ and $\prob_{x,z}^{U_N} := \prob_x^{U_N} \left( \cdot \left\vert X_{\tau_{\partial U_N} } = \floor{z} \right. \right)$.

Let $x_0 \in D$ and $z \in \partial D$ be a nice point.
Let $a \in (0,2)$ be a parameter measuring the thickness level,
\begin{equation}
\label{eq:c0}
g = \frac{2}{\pi}
\quad \mathrm{and} \quad
c_0 = \frac{2}{\pi} \Big( \gamma_{\rm EM} + \frac{1}{2} \log 8 \Big)
\end{equation}
be universal constants appearing in the asymptotic behaviour of the discrete Green function (see Lemma \ref{lem:Green}); here $\gamma_{\rm EM}$ stands for the Euler--Mascheroni constant.
We define a random Borel measure $\mu^{U,a}_{x_0;N}$ on $\C$ by setting for all Borel sets $A \subset \C$,
\begin{equation}
\label{eq:def_nu}
\mu^{U,a}_{x_0;N}(A)
:= \frac{\log N}{N^{2 - a}} \sum_{x \in \Z^2} \indic{x/N \in A} \indic{ \ell_x^{\tau_{\partial U_N}} \geq g a \log^2 N}
\mathrm{~under~} 
\prob^{U_N}_{Nx_0}.
\end{equation}
We also define the conditioned version $\mu^{U,a}_{x_0,z;N}$ of $\mu^{U,a}_{x_0;N}$ by replacing $\prob^{U_N}_{Nx_0}$ by $\prob^{U_N}_{Nx_0,Nz}$.

One of our main theorems is the following.

\begin{theorem}\label{th:convergence}
For all $a \in (0,2)$, the sequence
$\mu^{U,a}_{x_0;N}, N \geq 1,$
(resp. $\mu^{U,a}_{x_0,z;N}, N \geq 1$)
converges weakly relatively to the topology of weak convergence (resp. vague convergence) on $U$. Moreover, the limiting measure has the same distribution as $e^{c_0a/g} \Mc^{U,a}_{x_0}$ (resp. $e^{c_0a/g} \Mc^{U,a}_{x_0,z}$) built in \cite{bass1994,AidekonHuShi2018,jegoGMC}.
\end{theorem}

In Section \ref{subsec:Brownian_multiplicative_chaos}, we recall a precise definition of the above-mentioned Brownian multiplicative chaos measures $\Mc^{U,a}_{x_0}$ and $\Mc^{U,a}_{x_0,z}$.

We now emphasise the difficulties inherent to the random walk setting that are not present in the Brownian motion case considered in \cite{bass1994,jegoGMC}.
Theorem \ref{th:convergence} looks very similar to \cite[Theorem 1.1]{jegoGMC} (see also \cite{bass1994} for partial results) which studies flat measures $\Mc_\eps, \eps>0,$ supported on the set of thick points of planar Brownian motion. See Section \ref{subsec:Brownian_multiplicative_chaos} for more details about this. But let us emphasise that the approach of \cite{jegoGMC} cannot be adapted to prove Theorem \ref{th:convergence} and that a new strategy is needed. Indeed, the proof of \cite[Theorem 1.1]{jegoGMC} is based on the $L^1$-convergence of $(\Mc_\eps(A), \eps >0)$ for all Borel set $A \subset \C$. This strong form of convergence is crucial to the strategy in \cite{jegoGMC}. Here, it is not even a priori clear how to build the random measures $\mu^{U,a}_{x_0,z;N}, N \geq 1$, on the same probability space so that $(\mu^{U,a}_{x_0,z;N}(A), N\geq 1)$ converges in $L^1$. For instance, coupling the random walks via the same Brownian motion through KMT-type couplings does not seem to be tractable, or is at least too rough. As mentioned in the introduction, our proof of Theorem \ref{th:convergence} will rely on a characterisation of the law of Brownian multiplicative chaos, which we describe below.

We first mention however that Abe and Biskup \cite{AbeBiskup} have recently established a result with a similar flavour but important differences. Indeed, they consider a random walk on a box with wired boundary conditions (so it is uniformly resampled on the boundary every time it touches the boundary) and run the walk up to a time proportional to the cover time. In this regime, the local times of the walk are very closely related to the Gaussian free field and indeed their limiting measure is the Liouville measure (in contrast to here).

A direct consequence of Theorem \ref{th:convergence} is the convergence of the appropriately scaled number of random walk's thick points. This answers a question raised in \cite{jego2020} and considerably improves the previous known estimates on the fractal dimension \cite{dembo2001,rosen2006,BassRosen2007,jego2020} of the set of thick points. For $a \in (0,2)$, we denote
\[
\Tc_N(a) := \left\{ x \in \Z^2, \ell_x^{\tau_{\partial U_N}} \geq \frac{2}{\pi} a \log^2 N \right\}.
\]
Recalling the definition \eqref{eq:c0} of $g$ and $c_0$, we have:

\begin{corollary}\label{cor:thick _points}
For all $a \in (0,2)$, the following convergence holds in distribution: under $\prob^{U_N}_{Nx_0}$,
\[
\frac{\log N}{N^{2-a}} \# \Tc_N(a) \xrightarrow[N \to \infty]{} e^{c_0a/g} \Mc^{U,a}_{x_0}(U).
\]
Moreover, the limit is nondegenerate, i.e. $\Mc^{U,a}_{x_0}(U) \in (0,\infty)$ a.s.
\end{corollary}

As mentioned in \cite{jego2020}, despite the strong link between the local times and the GFF, this shows a subtle difference in the structure of thick points of random walk compared to those of the GFF which cannot be observed through rougher estimates such as the fractal dimension. Indeed, the analogue of Corollary \ref{cor:thick _points} with the local times replaced by half of the GFF squared uses a normalisation factor with $\sqrt{\log N}$ instead of $\log N$. See \cite{BiskupLouidor}.

\begin{remark}
To ease the exposition we decided to focus on the measures $\mu^{U,a}_{x_0;N}$ defined above, but one can consider random measures on $\C \times \R$ defined by: for $A \in \Bc(\bar{U})$ and $T \in \Bc(\R \cup \{+\infty\})$,
\[
\tilde{\mu}^{U,a}_{x_0;N}(A \times T) :=
\frac{\log N}{N^{2 - a}} \sum_{x \in \Z^2} \indic{x/N \in A} \indic{ \sqrt{\ell_x^{\tau_{\partial U_N}}} -  \sqrt{ga} \log N \in T}.
\]
Once the convergence of $\mu^{U,a}_{x_0;N}$ is established, it can be shown that $\tilde{\mu}^{U,a}_{x_0;N}, N \geq 1$, converges, relative to the topology of vague convergence on $\bar{U} \times (\R \cup \{+\infty\})$ to a product measure: $\Mc^{U,a}_{x_0}$ times an exponential measure. 
See \cite{jegoGMC} for the case of local times of Brownian motion.
\end{remark}

Finally, the convergence of thick points of random walk to Brownian multiplicative chaos opens the door to other scaling limit results. We mention the paper \cite{ABJL21} which builds and studies a multiplicative chaos associated to the so-called Brownian loop soup. When the intensity of the loop soup is critical, \cite{ABJL21} shows that the resulting chaos is closely related to Liouville measure elucidating connections between Brownian multiplicative chaos, Gaussian free field and Liouville measure. This identification of measures heavily relies on the scaling limit results of the current paper. A stronger form of convergence than what is stated in Theorem \ref{th:convergence} is actually needed in \cite{ABJL21}. This convergence is stated in Theorem \ref{th:extension} and is a by-product of our approach to Theorem \ref{th:convergence}. We preferred to defer the exposition of this result to Section \ref{sec:joint} because it requires the introduction of many more notations.

\subsection{Brownian multiplicative chaos: background and extension}\label{subsec:Brownian_multiplicative_chaos}

\paragraph*{Background}

This section recalls the definition of Brownian multiplicative chaos measure $\Mc^{U,a}_{x_0,z}$ as well as provides the extension of the results of \cite{bass1994,AidekonHuShi2018,jegoGMC} that we need. We follow the construction of \cite{jegoGMC} (see also \cite{bass1994} for partial results and \cite{AidekonHuShi2018} for a different construction). For a nice domain $U \subset \C$ and $x_0 \in U$, let $\prob^U_{x_0}$ be the law under which $(B_t, t \leq \tau_{\partial U})$ is a Brownian motion starting at $x_0$ and stopped at the first exit time of $U$:
\[
\tau_{\partial U} := \inf \left\{ t >0: B_t \in \partial U \right\}.
\]
For $x_0 \in U$ and a nice point $z \in \partial U$, we will also consider the conditional law $\prob_{x_0,z}^U := \prob_{x_0}^U \left( \cdot \left\vert B_{\tau_{\partial U}} = z \right. \right)$ which is rigorously defined for instance in \cite[Notation 2.1]{AidekonHuShi2018}.
For all $x \in U$ and $\eps >0$, define the local time $L_{x,\eps}$ of the circle $\partial D(x,\eps)$ up to time $\tau_{\partial U}$:
\[
L_{x,\eps} := \lim_{\substack{r \to 0\\r >0}} \frac{1}{2r} \int_0^{\tau_{\partial U}} \indic{\eps - r \leq \abs{B_t-x} \leq \eps + r} dt
\]
with the convention that $L_{x,\eps}=0$ if the disc $D(x,\eps)$ is not fully included in $U$.
\cite[Proposition 1.1]{jegoGMC} shows that these local times are well-defined for all $x \in U$ and $\eps >0$ simultaneously. For all parameter values $a \in (0,2)$ measuring the thickness level, we can thus define the random measure
\begin{equation}
\label{eq:def_measure_one_point}
A \in \Bc(\C) \mapsto \abs{\log \eps} \eps^{-a} \int_A \indic{ \frac{1}{\eps} L_{x,\eps} \geq 2a \abs{\log \eps}^2 } dx.
\end{equation}
\cite{jegoGMC} shows that for all $a \in (0,2)$ and under $\prob_{x_0,z}^U$, the previous measure converges in probability (relatively to the weak convergence) as $\eps \to 0$ to a nondegenerate random measure $\Mc^{U,a}_{x_0,z}$, our object of interest. Let us point out that this measure can also be constructed by exponentiating the square root of the local times $L_{x,\eps}$, justifying the name ``Brownian multiplicative chaos''. This random measure is conformally covariant and, almost surely, it is nondegenerate, supported on the set of thick points of Brownian motion and its carrying dimension equals $2-a$ (see e.g. \cite[Corollary 1.4]{jegoGMC}).

\paragraph*{Extension}

In this paper, a crucial new idea will be to consider the ``multipoint'' analogue of this measure. We will denote by $\Sc$ the collection of sets
\begin{equation}
\label{def:setSc}
\Dc\Xc\Zc = \left\{ (D_i,x_i,z_i), i =1 \dots r \right\}
\end{equation}
where $r \geq 1$, for all $i=1 \dots r$, $D_i$ is a nice domain, $x_i \in D_i$, $z_i \in \partial D_i$ is a nice boundary point, and the $z_i$'s are pairwise distinct points (i.e. $z_i \neq z_j$ for all $i \neq j$). If $\Dc\Xc\Zc \in \Sc$, we will (with some abuse of notations) see the set $\Dc\Xc\Zc$ as a triplet $\Dc, \Xc, \Zc$ of domains, starting points and exit points. We will for instance write ``$D \in \Dc$'' when we mean that we pick a domain that occurs in $\Dc\Xc\Zc$. Similarly, we will write $\Dc\Xc$ when we forget about the exit points.

We now define the multipoint analogue of $\Mc^{U,a}_{x_0,z}$. Let $\Dc\Xc\Zc = \{ (D_i,x_i,z_i), i = 1 \dots r\} \in \Sc$. For all $i=1 \dots r$, we consider independent Brownian motions distributed according to $\P_{x_i,z_i}^{D_i}$ and we denote by $L_{x,\eps}^{(i)}$ their associated local times. For all thickness level $a \in (0,2)$ and Borel set $A \subset \C$, we define
\[
\Mc^{\Dc,a}_{\Xc,\Zc;\eps}(A) :=
\abs{\log \eps} \eps^{-a} \int_A \indic{ \frac{1}{\eps} \sum_{i=1}^r L_{x,\eps}^{(i)} \geq 2a \abs{\log \eps}^2 } \indic{\forall i=1 \dots r, L_{x,\eps}^{(i)}>0} dx.
\]
We emphasise that, in this definition, the thick points arise from the interaction of the different trajectories. In particular, the single trajectories are not required to be $a$-thick. In fact, as we will see in Proposition \ref{prop:intersection_multipoint}, a single trajectory will typically be $\alpha$-thick where $\alpha$ is uniformly distributed in $[0,a]$.
Note also that the normalisation is the same as the individual measures \eqref{eq:def_measure_one_point}. This indicates that they contribute in the same manner to the occurrence of thick points.

A rather simple modification of \cite[Theorem 1.1]{jegoGMC} shows:

\begin{proposition}\label{prop:def_measures}
For all $a \in (0,2)$, relative to the topology of weak convergence, the sequence of random measures $\Mc^{\Dc,a}_{\Xc,\Zc;\eps}$ converges as $\eps \to 0$ to some random measure $\Mc^{\Dc,a}_{\Xc,\Zc}$ in probability.
\end{proposition}

The proof of this result is contained in Appendix \ref{app:intersection}. Let us comment that $\Mc^{\Dc,a}_{\Xc,\Zc}$ clearly vanishes almost surely if $\bigcap_{i=1}^r D_i = \varnothing$. Section \ref{subsec:further_results} investigates some further properties of this multipoint version of Brownian multiplicative chaos. In particular, we explain that we can express $\Mc^{\Dc,a}_{\Xc,\Zc}$ in terms of the integral of the ``intersection'' of one-point Brownian multiplicative chaos measures
\[
\bigcap_{i=1}^r \Mc_{x_i,z_i}^{D_i,a_i}.
\]
This ``intersection measure'' is a natural measure supported on the intersection of the set of thick points associated to each single Brownian motion with suitable thickness level. Further surprising properties of these measures are discussed in Section \ref{subsec:further_results}. See in particular Proposition \ref{prop:intersection_multipoint}.

Finally, we will consider the process of measures $\left (\Mc^{\Dc,a}_{\Xc,\Zc}, \Dc\Xc\Zc \in \Sc \right)$. We have already defined the one-dimensional marginals of this process. The definition of the finite-dimensional marginals is done in the following way: if $\Dc_j\Xc_j\Zc_j \in \Sc, j =1 \dots J$, for all $(D,x_0,z)$ appearing in one of the $\Dc_j\Xc_j\Zc_j$, we always use the same Brownian motion from $x_0$ to $z$ to define the measures $\Mc_{\Xc_j,\Zc_j}^{\Dc_j,a}$.
As before, for different triplets $(D,x_0,z)$, we use independent Brownian motions.
In particular, if $\Dc\Xc\Zc \cap \Dc'\Xc'\Zc' = \varnothing$, the measures $\Mc^{\Dc,a}_{\Xc,\Zc}$ and $\Mc_{\Xc',\Zc'}^{\Dc',a}$ are independent. This definition is consistent and thus uniquely defines the process $\left (\Mc^{\Dc,a}_{\Xc,\Zc}, \Dc\Xc\Zc \in \Sc \right)$. We mention that we will sometimes write $\Mc_{\Dc\Xc\Zc}^a$ instead of $\Mc_{\Xc,\Zc}^{\Dc,a}$ to clarify the situation.

\subsection{Characterisation of Brownian multiplicative chaos}
\label{subsec:charac}

We can now state our characterisation of the law of Brownian multiplicative chaos.
We start off by introducing some complex analysis notations. Let $a \in (0,2)$ be a thickness level.
For any nice domain $D \subset \C$, $x \in D$ and a nice point $z \in \partial D$, we will denote by $\CR(x,D)$ the conformal radius of $D$ seen from $x$, $G^D$ the Green function of $D$ with zero boundary conditions and $H^D(x,z)dz = \PROB{x}{B_{\tau_{\partial D}} \in dz}$ the Poisson kernel or harmonic measure of $D$. See Section \ref{subsec:notations} for precise definitions.
We set
\begin{equation}
\label{eq:def_psi}
\psi^{D,a}_{x_0,z}(x) := \CR(x,D)^a G^D(x_0,x) \frac{H^D(x,z)}{H^D(x_0,z)}.
\end{equation}
By convention, we will set $\psi^{D,a}_{x_0,z}(x) = 0$ if $x \notin D$.
We also introduce, for any $r \geq 1$, a notation for the $(r-1)$-dimensional simplex
\begin{equation}
\label{eq:def_simplex}
E(a,r) := \left\{ \mathsf{a} = (a_1, \dots, a_r) \in (0,a]^r: a_1 + \dots + a_r = a \right\}.
\end{equation}
The Lebesgue measure on $E(a,r)$ will be denoted by $d \mathsf{a} = da_1 \dots da_{r-1}$.

We are about to consider properties characterising the law of the process $\left (\Mc^{\Dc,a}_{\Xc,\Zc}, \Dc\Xc\Zc \in \Sc \right)$ defined in Section \ref{subsec:Brownian_multiplicative_chaos}. The most important one will be the spatial Markov property (Property \ref{charac2}). Because it will be notationally heavy, we first present a simple particular case of it which explains the main idea. Let $\Dc\Xc\Zc \in \Sc$ be of the form $\Dc\Xc\Zc=\{(D,x_0,z)\}$. Let $D'$ be a nice subset of $D$ containing $x_0$. Then Property \ref{charac2} amounts to saying that:
\begin{equation}
\label{eq:baby_Markov}
\Mc^{D,a}_{x_0,z}
\quad \mathrm{and} \quad
\Mc_{x_0,Y}^{D',a} + \Mc_{Y,z}^{D,a} + \Mc_{(D',x_0,Y),(D,Y,z)}^{a}
\end{equation}
have the same law, where $Y$ has the law of $B_{\tau_{\partial D'}}$ under $\prob_{x_0,z}^D$. Conditionally on $\{ Y=y \}$, the joint law of the measures $\Mc_{x_0,Y}^{D',a}$, $\Mc_{Y,z}^{D,a}$ and $\Mc_{(D',x_0,Y),(D,Y,z)}^{a}$ is by definition that of $\Mc_{x_0,y}^{D',a}$, $\Mc_{y,z}^{D,a}$ and $\Mc_{(D',x_0,y),(D,y,z)}^{a}$. This property comes from the following simple observation. Let $(B_t, t \leq \tau_{\partial D})$ be a Brownian motion in $D$ starting at $x_0$ and conditioned to exit $D$ through $z$. We divide $(B_t, t \leq \tau_{\partial D})$ into $(B_t, t\leq \tau_{\partial D'})$ and $(B_t, \tau_{\partial D'} \leq t \leq \tau_{\partial D})$. An $a$-thick point for the overall trajectory is either entirely generated by one of the two small trajectories and missed by the other one, or comes from the intersection of both.

\medskip

We now explain our characterisation. Let $\left (\mu^{\Dc,a}_{\Xc,\Zc}, \Dc\Xc\Zc \in \Sc \right)$ be a stochastic process taking values in the set of finite Borel measures.
We consider the following properties:
\begin{enumerate}[label=(\subscript{P}{{\arabic*}})]
\item(Average value)
\label{charac1}
For all $\Dc\Xc\Zc = \left\{ (D_i,x_i,z_i), i=1 \dots r \right\} \in \Sc$ and for all Borel set $A \subset \C$,
\begin{align*}
& \Expect{ \mu^{\Dc,a}_{\Xc,\Zc}(A) }
= \int_A dx \int_{\mathsf{a} \in E(a,r)} d \mathsf{a} \prod_{k=1}^r \psi_{x_k,z_k}^{D_k,a_k}(x).
\end{align*}
\item(Markov property)
\label{charac2}
Let $\Dc\Xc\Zc \in \Sc$, $(D,x_0,z) \in \Dc\Xc\Zc$ and let $D'$ be a nice subset of $D$ containing $x_0$. Let $Y$ be distributed according to $B_{\tau_{\partial D'}}$ under $\P_{x_0,z}^D$. The joint law of $(\mu_{\Xc',\Zc'}^{\Dc',a}, \Dc'\Xc'\Zc' \subset \Dc\Xc\Zc)$ is the same as the joint law given by for all $\Dc'\Xc'\Zc' \subset \Dc\Xc\Zc$,
\[
\left\{
\begin{array}{l}
\mu_{\Xc',\Zc'}^{\Dc',a}
\mathrm{~if~} (D,x_0,z) \notin \Dc'\Xc'\Zc', \\
\mu_{\bar{\Dc}\bar{\Xc}\bar{\Zc} \cup \{(D',x_0,Y)\} }^a
+ \mu_{\bar{\Dc}\bar{\Xc}\bar{\Zc} \cup \{(D,Y,z)\} }^a
+ \mu_{\bar{\Dc}\bar{\Xc}\bar{\Zc} \cup \{(D',x_0,Y), (D,Y,z)\} }^a
\mathrm{~otherwise},
\end{array}
\right.
\]
where in the second line we denote $\bar{\Dc}\bar{\Xc}\bar{\Zc} = \Dc'\Xc'\Zc' \backslash \{(D,x_0,z)\}$.
\item(Independence)
\label{charac2bis}
For all disjoint sets $\Dc\Xc\Zc, \Dc'\Xc'\Zc' \in \Sc$, the measures $\mu^{\Dc,a}_{\Xc,\Zc}$ and $\mu_{\Xc',\Zc'}^{\Dc',a}$ are independent.
\item(Non-atomicity)
\label{charac3}
For all $\Dc\Xc\Zc \in \Sc$, with probability one, simultaneously for all $x \in \C$, $\mu^{\Dc,a}_{\Xc,\Zc}(\{x\}) = 0$.
\end{enumerate}

\begin{theorem}\label{th:charac}
Let $a \in (0,2)$. The process
$\left (\Mc^{\Dc,a}_{\Xc,\Zc}, \Dc\Xc\Zc \in \Sc \right)$
from Section \ref{subsec:Brownian_multiplicative_chaos} satisfies Properties \ref{charac1}-\ref{charac3}. Moreover, if
$\left (\mu^{\Dc,a}_{\Xc,\Zc}, \Dc\Xc\Zc \in \Sc \right)$
is another process taking values in the set of finite Borel measures satisfying Properties \ref{charac1}-\ref{charac3}, then it has the same law as
$\left (\Mc^{\Dc,a}_{\Xc,\Zc}, \Dc\Xc\Zc \in \Sc \right)$.
\end{theorem}

Biskup and Louidor \cite{BiskupLouidor} provide a somewhat similar characterisation of the Liouville measure. The main difference is that Properties \ref{charac2} and \ref{charac2bis} are replaced by how the spatial Markov property of the Gaussian free field translates to the Liouville measure.

Other characterisations have been formulated before: let $D$ be a fixed nice domain, $x_0 \in D$, $z \in \partial D$ nice and consider the pair given by the measure $\Mc^{D,a}_{x_0,z}$ together with the Brownian motion $(B_t, t \leq \tau_{\partial D})$ from which it has been built. Then the pair $(\Mc^{D,a}_{x_0,z},B)$ is uniquely characterised by

$\bullet$ the measurability of $\Mc^{D,a}_{x_0,z}$ with respect to the Brownian path $B$,

$\bullet$ the way the law of the path $B$ is changed given a sample of $\Mc^{D,a}_{x_0,z}$.
\newline
See Theorem 5.2 of \cite{bass1994}. See also Proposition \ref{prop:other_charac} for an extension of this characterisation to finitely many trajectories. The advantage of this characterisation is that it considers only one domain, with given starting and ending points and does not need to rely on the multipoint version of Brownian multiplicative chaos. But its drawback is that it refers explicitly to the underlying Brownian motion and it seems to be less applicable in practice. For instance, in the context of our application to random walk, it does not seem easy to apply this characterisation (even measurability is not a priori clear).

Let us also mention that the proof of Theorem \ref{th:charac} provides a construction of $\Mc^{D,a}_{x_0,z}$ through a martingale approximation (see Lemma \ref{lem:charac}). This is very similar to some aspects of the construction of \cite{AidekonHuShi2018} except that they divide the domain into small dyadic squares rather than long narrow rectangles. This might seem to be a cosmetic difference but it is in fact significant since it leads to a decomposition of the Brownian path into excursions from internal to boundary point rather than from boundary to boundary. This is at the heart of what leads to the recursive decomposition of the proof and in turn to the theorem, since the measure $\Mc^{D,a}_{x_0,z}$ is also itself of this type.

Finally, it is possible that Properties \ref{charac1}-\ref{charac2bis} are enough to characterise the law, but Property \ref{charac3} is necessary for our current proof; see especially Lemma \ref{lem:charac3old}. In practice, in our context of uniform measure on thick points of random walk, Property \ref{charac3} is a consequence of uniform-integrability-type estimates that are needed in order to verify Property \ref{charac1}.

\subsection{Further results on multipoint Brownian multiplicative chaos}\label{subsec:further_results}

In this section, we study in greater detail the multipoint version of Brownian multiplicative chaos measures.
We start by introducing the ``intersection'' of Brownian multiplicative chaos measures: a measure whose support is included in the intersection of the support of each intersected measure.
Let $\Dc\Xc\Zc = \{(D_i,x_i,z_i), i=1 \dots r \} \in \Sc$ and consider independent Brownian motions $B_{x_i,z_i}^{D_i}$ distributed according to $\P_{x_i,z_i}^{D_i}$ for all $i=1 \dots r$. Denote by $L_{x,\eps}^{(i)}$ their associated local times.
Let $a_i >0, i =1 \dots r,$ be thickness levels such that $a := \sum a_i <2$. We now consider the measure defined by: for all Borel set $A \subset \C$,
\[
\bigcap_{i=1}^{r} \Mc_{x_i,z_i;\eps}^{D_i,a_i} (A) := \abs{\log \eps}^r \eps^{-a} \int_A \prod_{i=1}^r \indic{ \frac{1}{\eps} L^{(i)}_{x,\eps} \geq 2 a_i \abs{\log \eps}^2} dx.
\]
Proposition \ref{prop:intersection} below studies the limit of these measures and Proposition \ref{prop:intersection_multipoint} studies the link between this limiting measure and $\Mc^{\Dc,a}_{\Xc,\Zc}$ introduced in Section \ref{subsec:Brownian_multiplicative_chaos}. These results are proven in Appendix \ref{app:intersection}.

\begin{proposition}\label{prop:intersection}
\begin{enumerate}[label=(\roman*)]
\item
Relative to the topology of weak convergence, the measure $\bigcap_{i=1}^r \Mc_{x_i,z_i;\eps}^{D_i,a_i}$ converges as $\eps \to 0$ towards a random finite Borel measure $\bigcap_{i=1}^r \Mc_{x_i,z_i}^{D_i,a_i}$ in probability.
\item
Inductive decomposition.
If $r \geq 2$, the sequence of random Borel measures
\begin{equation}
\label{eq:prop_intersection}
A \in \Bc(\C) \mapsto \abs{\log \eps} \eps^{-a_r} \int_A \indic{\frac{1}{\eps} L^{(r)}_{x,\eps} \geq 2 a_r \abs{\log \eps}^2} \bigcap_{i=1}^{r-1} \Mc_{x_i,z_i}^{D_i,a_i}(dx)
\end{equation}
converges as $\eps \to 0$ to $\bigcap_{i=1}^r \Mc_{x_i,z_i}^{D_i,a_i}$ in probability, relative to the topology of weak convergence.
\item
The measure $\bigcap_{i=1}^r \Mc_{x_i,z_i}^{D_i,a_i}$ is measurable with respect to $\sigma \left( \Mc_{x_i,z_i}^{D_i,a_i}, i =1 \dots r \right)$, the underlying sigma-algebra being the one associated to the topology of weak convergence.
\item
For all $A \in \Bc(\C)$,
\[
\Expect{ \bigcap_{i=1}^r \Mc_{x_i,z_i}^{D_i,a_i}(A) } = \int_A \prod_{i=1}^r \psi_{x_i,z_i}^{D_i,a_i}(x) dx.
\]
\item
With probability one, simultaneously for all Borel set $A$ of Hausdorff dimension strictly smaller than $2-\sum_{i=1}^r a_i$, $\bigcap_{i=1}^r \Mc_{x_i,z_i}^{D_i,a_i} (A) = 0$.
\item
The stochastic process
\[
(a_i)_{i = 1 \dots r} \in \{ (\alpha_i)_{i =1 \dots r} \in (0,2)^r: \sum \alpha_i < 2 \} \mapsto \bigcap_{i=1}^r \Mc_{x_i,z_i}^{D_i,a_i}
\]
taking values in the set of finite Borel measures, equipped with the topology of weak convergence, possesses a measurable modification.
\end{enumerate}
\end{proposition}

For the following proposition, we consider the measure $\Mc^{\Dc,a}_{\Xc,\Zc}$ built from the same Brownian motions as the ones used to defined the previous intersection measures.

\begin{proposition}[Disintegration]\label{prop:intersection_multipoint}
Let $a \in (0,2)$. If $r \geq 2$, then
\[
\Mc^{\Dc,a}_{\Xc,\Zc}
= \int_{\mathsf{a} \in E(a,r)} d \mathsf{a} \bigcap_{k=1}^r 
\Mc_{x_k,z_k}^{D_k,a_k}
\quad \quad \mathrm{a.s.}
\]
\end{proposition}

Note that the integral of intersection measures above is well-defined thanks to Proposition \ref{prop:intersection}, Point \textit{(vi)}.

This result can be compared to the disintegration theorem in measure theory. In words, this proposition shows that the measure $\Mc^{\Dc,a}_{\Xc,\Zc}$ ``restricted to the event'' that, for all $k=1 \dots r$, the contribution of the $k$-th trajectory to the overall thickness $a$ is exactly $a_k$, agrees with the intersection measure $\bigcap_{k=1}^r \Mc_{x_k,z_k}^{D_k,a_k}$. With the standard disintegration theorem, one is able to make sense of the disintegrated measure for almost every $\mathsf{a} \in E(a,r)$. Here, the randomness of the measures helps us and we are able to make sense of these measures almost surely, simultaneously for all $\mathsf{a} \in E(a,r)$.

In view of Proposition \ref{prop:intersection_multipoint}, we can rewrite Property \ref{charac2} in the following way. Let $D' \subset D$ be two nice domains, $x_0 \in D'$ and $z \in \partial D$ be a nice point, then
\[
\Mc^{D,a}_{x_0,z} = \Mc_{x_0,Y}^{D',a} + \Mc_{Y,z}^{D,a} + \int_0^a \Mc_{x_0,Y}^{D',a-\alpha} \cap \Mc_{Y,z}^{D,\alpha} d \alpha
\]
with $Y = B_{\tau_{\partial D'}}$. A surprising consequence of Proposition \ref{prop:intersection_multipoint} is the following.

For all $x \in D'$, if we condition $x$ to be an $a$-thick point for the overall trajectory $(B_t, t \leq \tau_{\partial D})$ and if we condition the two small trajectories $(B_t, t \leq \tau_{\partial D'})$ and $(B_t, \tau_{\partial D'} \leq t \leq \tau_{\partial D})$ to visit $x$, then the thickness level of $x$ for one of the two small trajectories will be uniformly distributed in $(0,a)$. Proposition \ref{prop:intersection_multipoint} makes this statement formal. Indeed, by definition, this type of thick points is described by the measure
\[
\Mc_{(D',x_0,Y),(D,Y,z)}^{a} = \int_0^a \Mc_{x_0,Y}^{D',a-\alpha} \cap \Mc_{Y,z}^{D,\alpha} d \alpha
\]
where the equality follows from Proposition \ref{prop:intersection_multipoint}. On the right hand side, the thickness associated to each subtrajectory is fixed to $a-\alpha$ and $\alpha$ respectively where $\alpha$ is sampled uniformly in $[0,a]$.

\medskip

We finish this section by giving an intrinsic characterisation of the intersection measure $\bigcap_{i=1}^r \Mc_{x_i,z_i}^{D_i,a_i}$. The characterisation below is a simple extension of the characterisation of the multiplicative chaos associated to one Brownian trajectory, but it is nevertheless an important result since it allows one to quickly identify the measure.

The next result uses the notations introduced above Proposition \ref{prop:intersection}. In particular, recall that $B_{x_i,z_i}^{D_i}$ denotes the Brownian motion distributed according to $\P_{x_i,z_i}^{D_i}$ associated to $\bigcap_{i=1}^r \Mc_{x_i,z_i}^{D_i,a_i}$. For all $i = 1\dots r$, we view $B_{x_i,z_i}^{D_i}$ as a random element of the set $\Pc$ of càdlàg paths in $\R^2$ with finite durations. See Section \ref{sec:joint} for details, in particular concerning the topology associated to $\Pc$. The following proposition describes the law of the Brownian paths after shifting the probability measure by $\bigcap_{i=1}^r \Mc_{x_i,z_i}^{D_i,a_i}(dx)$ (the so-called rooted measure). As we will see, the resulting trajectories can be written as the concatenations of three independent pieces $B_{x_i,x}^{D_i} \wedge \Xi_x^{D_i,a_i} \wedge B_{x,z_i}^{D_i}$. The first one is a trajectory $B_{x_i,x}^{D_i}$ with law $\P_{x_i,x}^{D_i}$, i.e. a Brownian path conditioned to visit $x$ before exiting $D_i$. The second part $\Xi_x^{D_i,a_i}$ consists in the concatenation of infinitely many loops rooted at $x$ that are distributed according to a Poisson point process with intensity $a_i \nu_{D_i}(x,x)$. Here $\nu_{D_i}(x,x)$ is a measure on Brownian loops that stay in $D_i$ (see e.g. (2.12) in \cite{AidekonHuShi2018}). Finally, the last part of the trajectory is a Brownian motion $B_{x,z_i}^{D_i}$ distributed according $\P_{x,z_i}^{D_i}$, that is, a trajectory which starts at $x$ and which is conditioned to exit $D_i$ through $z_i$.

\begin{proposition}\label{prop:other_charac}
Let $F : \C \times \Pc^r \to \R$ be a bounded measurable function. Then
\begin{align}
\label{eq:rooted_measure}
\Expect{ \int_{\C} F(x, (B_{x_i,z_i}^{D_i})_{i = 1 \dots r}) \bigcap_{i=1}^r \Mc_{x_i,z_i}^{D_i,a_i}(dx) }
& = \int_{\cap_i D_i} \prod_{i=1}^r \psi_{x_i,z_i}^{D_i,a_i}(x) \\
\nonumber
& \times \Expect{ F(x, ( B_{x_i,x}^{D_i} \wedge \Xi_x^{D_i,a_i} \wedge B_{x,z_i}^{D_i} )_{i=1 \dots r} )} dx.
\end{align}
Moreover, if $\mu$ is another random Borel measure which is measurable w.r.t. $B_{x_i,z_i}^{D_i}, i =1 \dots r$, and which satisfies \eqref{eq:rooted_measure} for all bounded measurable function $F$, then $\mu = \bigcap_{i=1}^r \Mc_{x_i,z_i}^{D_i,a_i}$ almost surely.
\end{proposition}

As already alluded to, this type of characterisation is of little help when one wants to establish scaling limit results since it relies on the measurability of the underlying Brownian trajectories.

Finally, we mention that,
using Proposition \ref{prop:intersection_multipoint} and Proposition \ref{prop:other_charac} above, one can also compute the left hand side of \eqref{eq:rooted_measure} where the intersection measure has been replaced by the multipoint measure $\Mc_{\Xc,\Zc}^{\Dc,a}$.
Therefore, a similar characterisation concerning the multipoint measure could also be stated.

\subsection{Outline of proofs}\label{subsec:organisation}

We now present the organisation of the paper and explain the main ideas behind the proofs of Theorems \ref{th:convergence} and \ref{th:charac}.

Section \ref{sec:characterisation} is devoted to the proof of Theorem \ref{th:charac}. It will start by proving that Brownian multiplicative chaos satisfies Properties \ref{charac1}-\ref{charac3} assuming Propositions \ref{prop:def_measures}, \ref{prop:intersection} and \ref{prop:intersection_multipoint} on the multipoint version of Brownian multiplicative chaos. These propositions will be proven in Appendix \ref{app:intersection}. The rest of Section \ref{sec:characterisation} will deal with the uniqueness part of Theorem \ref{th:charac} and we now sketch its proof.
Let
$\left (\mu^{\Dc,a}_{\Xc,\Zc}, \Dc\Xc\Zc \in \Sc \right)$
be a process of Borel measures satisfying Properties \ref{charac1}-\ref{charac3}.
Let $D$ be a nice domain, $x_0 \in D$ and a nice point $z \in \partial D$. We are going to explain the characterisation of the law of $\mu^{D,a}_{x_0,z}$. The characterisation of the law of more general marginals follows along the same lines. The only extra difficulty lies in the notations.
We will start by noticing that Property \ref{charac3} implies that we can find a deterministic direction such that almost surely all the lines parallel to this direction are not seen by the measure $\mu^{D,a}_{x_0,z}$. Without loss of generality, assume that this direction is the vertical one (straightforward adaptations would need to be made in the case of a general direction).
We will slice the domain $D$ into many narrow rectangle-type domains $D \cap (q2^{-p}, (q+2)2^{-p}) \times \R$, $q \in \Z$.
By iterating Property \ref{charac2}, we will be able to decompose
\[
\mu^{D,a}_{x_0,z}
\overset{\mathrm{(d)}}{=}
\sum_{\Dc\Xc\Zc \subset \{ (D^p_i,x^p_i,x^p_{i+1}), i \leq I_p-1\} } \mu^{\Dc,a}_{\Xc,\Zc}.
\]
$D_i^p$ will be a narrow rectangle as above centred at $x_i^p$ and $x_i^p,i \geq 1,$ will correspond to the successive hitting points of $2^{-p}\Z \times \R$ of a Brownian trajectory.
See \eqref{eq:proof_thm_charac} for precise definitions and Figure \ref{fig1} for an illustration of these successive hitting points. The idea is then that most of the randomness comes from the points $x_i^p, i \geq 1$, and we do not change the measure so much by replacing each term
\[
\mu^{\Dc,a}_{\Xc,\Zc}
\mathrm{~by~}
\Expect{ \left. \mu^{\Dc,a}_{\Xc,\Zc} \right\vert x_i^p, p \geq 1}.
\]
This latter expression is entirely determined by Property \ref{charac1} and does not depend on the process
$\left (\mu^{\Dc,a}_{\Xc,\Zc}, \Dc\Xc\Zc \in \Sc \right)$
any more. This conditional expectation encodes a lot of information. For instance, it ensures the measure to be concentrated around the Brownian trajectory. In fact, it provides a martingale approximation of the measure $\mu^{D,a}_{x_0,z}$ as we will see in Lemma \ref{lem:charac}.
The proof will then consist in showing that the error in the above approximation tends to zero when $p \to \infty$. The fact that almost surely $\mu^{D,a}_{x_0,z}$ gives zero-mass to any vertical line will be useful for this purpose making sure that we decomposed the initial measure into many small pieces.

\begin{figure}[ht]
   \centering
    \def\svgwidth{0.5\columnwidth}
   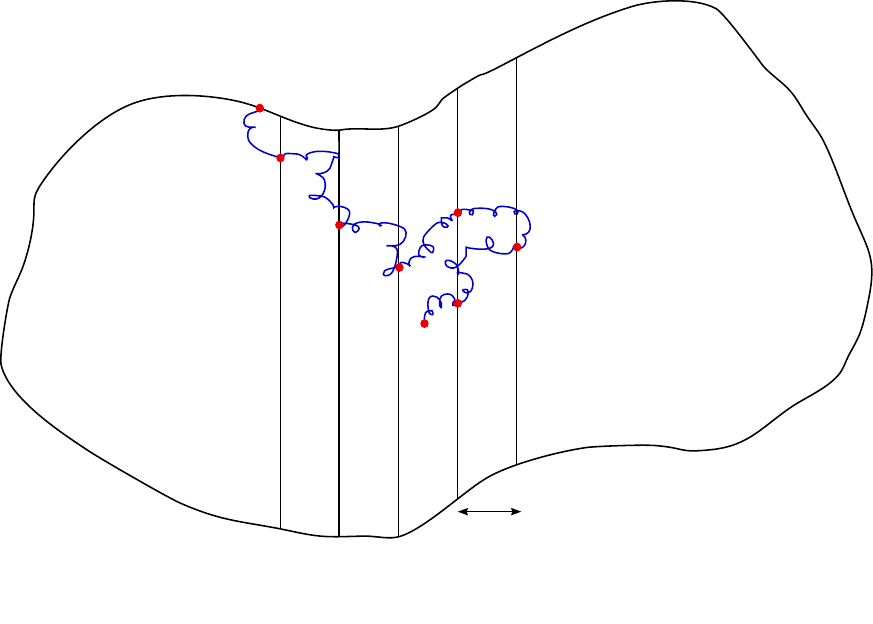
   \caption{Representation of the successive hitting points of $2^{-p} \Z \times \R$ by a Brownian motion sampled according to $\P_{x_0,z}^D$.}\label{fig1}
\end{figure}

\medskip

We now turn to the random walk part.
We will first show the convergence of $\mu^{U,a}_{x_0,z;N}$. The convergence of the unconditioned measures $\mu^{U,a}_{x_0;N}$ will then follow fairly quickly thanks to the weak convergence of the discrete Poisson kernel.
To show the convergence of $\mu^{U,a}_{x_0,z;N}$, the overall strategy is simple: we will prove that this sequence is tight and we will then identify the subsequential limits. The tightness is the easy part and relies on a first moment computation. Section \ref{sec:tight} is devoted to it. The identification of the subsequential limits uses Theorem \ref{th:charac} and is done in Section \ref{sec:identification}.
We sketch the main steps of this identification.
Let $x_* \in U$ and $z_* \in \partial U$ be a nice point. Let $(N_k,k \geq 1)$ be an increasing sequence of integers so that $(\mu^{U,a}_{x_*,z_*;N_k},k \geq 1)$ converges. In Lemma \ref{lem:uncountable_extraction}, we will show that we can extract a further subsequence $(N_k',k \geq 1)$ of $(N_k,k \geq 1)$ such that
for all $\Dc'\Xc'\Zc' \in \Sc$,
\[
\left( \mu_{\Xc,\Zc;N_k'}^{\Dc,a}, \Dc\Xc\Zc \subset \Dc'\Xc'\Zc' \right)
\]
converges. The above measures are the discrete analogue of the multipoint versions of Brownian multiplicative chaos and are defined in \eqref{eq:def_multipoint_discrete}. We denote by 
$ (\mu^{\Dc,a}_{\Xc,\Zc}, \Dc\Xc\Zc \in \Sc) $
the limiting process of finite Borel measures. Showing that we can extract such a subsequence requires some work since we consider an uncountable number of sequences. Thanks to Theorem \ref{th:charac}, to conclude the identification of the limiting measure $\mu^{U,a}_{x_*,z_*}$, it is then enough to show that the process
$ (\mu^{\Dc,a}_{\Xc,\Zc}, \Dc\Xc\Zc \in \Sc) $
satisfies Properties \ref{charac1}-\ref{charac3}. This will roughly follow along the same lines as in the Brownian case. In particular, the uniform integrabilitiy of $\mu_{\Xc,\Zc;N}^{\Dc,a}(\Z^2)$, $N \geq 1$, which is the content of Proposition \ref{prop:uniform_integrability} is key. This comes from a careful truncated second moment estimate which is similar to what was done in \cite{jegoGMC}. The proof of Proposition \ref{prop:uniform_integrability} is written in Section \ref{sec:uniform integrability}.

\subsection{Some notations}\label{subsec:notations}

We finish this introduction with some notations that will be used throughout the paper.
Let $D \subset \C$ be a nice domain. For $x \in D$ and a nice point $z \in \partial D$, we will denote by $\CR(x,D)$ the conformal radius of $D$ seen from $x$, $G^D$ the Green function of $D$ with zero boundary conditions normalised so that $G^D(x,y) \sim - \log \abs{x-y}$ as $\abs{x-y} \to 0$ and $H^D(x,z)dz = \PROB{x}{B_{\tau_{\partial D}} \in dz}$ the Poisson kernel or harmonic measure of $D$. These three quantities can be expressed in terms of a conformal map $f_D : D \to \D$ onto the unit disc (see e.g. \cite[Chapter 2]{LawlerBook05}): for all $x, y \in D$ and for all nice point $z \in \partial D$,
\begin{align}
\label{eq:conformal_radius}
\CR(x,D) & = \frac{1-\abs{f_D(x)}^2}{\abs{f_D'(x)}}, \\
\label{eq:Green_conformal}
G^D(x,y) & = \log \frac{\abs{1 - f_D(x) \overline{f_D(y)}}}{\abs{f_D(y) - f_D(x)}},\\
\label{eq:Poisson_conformal}
H^D(x,z) & = \abs{f_D'(z)} \frac{1-\abs{f_D(x)}^2}{2\pi\abs{f_D(x) - f_D(z)}^2}.
\end{align}
With the notations of Section \ref{subsec:RW}, we will similarly denote by $G^{D_N}$ and $H^{D_N}$ the discrete Green's function and Poisson kernel defined by: for all $x,y \in \Z^2$,
\begin{equation}
\label{eq:def_Green_Poisson}
G^{D_N}(x,y) := \EXPECT{x}{\ell_y^{\tau_{\partial D_N}}}
\mathrm{~and~}
H^{D_N}(x,y) := \prob^{D_N}_x \left( X_{\tau_{\partial D_N}} = y \right).
\end{equation}

In the rest of paper, $a \in (0,2)$ will always denote the thickness level that we look at.

\section{Characterisation: proof of Theorem \ref{th:charac}}\label{sec:characterisation}

We start by proving that Brownian multiplicative chaos satisfies Properties \ref{charac1}-\ref{charac3}.

\begin{proof}[Proof of Theorem \ref{th:charac}, existence]
Property \ref{charac1}, resp. \ref{charac3}, is a direct consequence of Proposition \ref{prop:intersection} \textit{(iv)}, resp. \textit{(v)}, and Proposition \ref{prop:intersection_multipoint}. Property \ref{charac2bis} follows from the fact that we consider independent Brownian motions.

We now prove Property \ref{charac2}. To ease notations, we will only prove this in the simplest case $\Dc\Xc\Zc=\{(D,x_0,z)\}$. The general case follows along the same lines. Let $D'$ be a nice subset of $D$ containing $x_0$. Let $B$ be a Brownian motion under $\prob_{x_0,z}^D$, $L_{x,\eps}$ its associated local times and let $L_{x,\eps}^{(0)}$ be the local times of $B$ stopped at the first exit time of $D'$ and $L^{(1)}_{x,\eps} := L_{x,\eps} - L_{x,\eps}^{(0)}$. We can write
\begin{align}
\label{eq:proof_P2}
\Mc_{x_0,z;\eps}^{D,a}(dx) & = \abs{\log \eps} \eps^{-a} \indic{\frac{1}{\eps}L_{x,\eps} \geq 2a |\log \eps|^2} dx
 = \abs{\log \eps} \eps^{-a} \left( \indic{\frac{1}{\eps}L_{x,\eps}^{(0)} \geq 2a |\log \eps|^2} \indic{L_{x,\eps}^{(1)} =0} \right. \\
 & ~~~~~~ \left. + \indic{\frac{1}{\eps}L_{x,\eps}^{(1)} \geq 2a |\log \eps|^2} \indic{L_{x,\eps}^{(0)} =0} + \indic{\frac{1}{\eps} \left( L_{x,\eps}^{(0)} + L_{x,\eps}^{(1)} \right) \geq 2a |\log \eps|^2, L_{x,\eps}^{(0)}>0, L_{x,\eps}^{(1)}>0}  \right) dx. \nonumber
\end{align}
If we denote by $Y$ the first hitting point of $\partial D'$ of the Brownian trajectory $B$, Proposition \ref{prop:def_measures} shows that the last term on the right hand side converges in probability towards $\Mc_{(D',x_0,Y),(D,Y,z)}^a$. We are now going to argue that the first right hand side term converges in probability towards $\Mc_{x_0,Y}^{D',a}$. Indeed, for all Borel set $A \subset \C$,
\begin{align*}
& \Expect{ \abs{ \Mc_{x_0,Y;\eps}^{D',a}(A) - |\log \eps| \eps^{-a} \int_A \indic{\frac{1}{\eps}L_{x,\eps}^{(0)} \geq 2a |\log \eps|^2} \indic{L_{x,\eps}^{(1)} =0} dx } } \\
& = |\log \eps| \eps^{-a} \int_A \prob_{x_0,z}^D \left( \frac{1}{\eps}L_{x,\eps}^{(0)} \geq 2a |\log \eps|^2, L_{x,\eps}^{(1)} >0 \right) dx.
\end{align*}
We can dominate
\[
\sup_\eps |\log \eps| \eps^{-a} \prob_{x_0,z}^D \left( \frac{1}{\eps}L_{x,\eps}^{(0)} \geq 2a |\log \eps|^2, L_{x,\eps}^{(1)} >0 \right)
\leq \sup_\eps |\log \eps| \eps^{-a} \prob_{x_0,z}^D \left( \frac{1}{\eps}L_{x,\eps} \geq 2a |\log \eps|^2 \right)
\]
which is integrable (see \eqref{eq:proof_th_intersection1}). Moreover, for all $x \notin \partial D'$,
\begin{align*}
& |\log \eps| \eps^{-a} \prob_{x_0,z}^D \left( \frac{1}{\eps}L_{x,\eps}^{(0)} \geq 2a |\log \eps|^2, L_{x,\eps}^{(1)} >0 \right) \\
& = \expect_{x_0,z}^D \left[ |\log \eps| \eps^{-a} \prob_{x_0,Y}^{D'} \left( \frac{1}{\eps}L_{x,\eps}^{(0)} \geq 2a |\log \eps|^2 \right) \prob_{Y,z}^D \left( L_{x,\eps}^{(1)} >0 \right) \right]
\end{align*}
tends to zero as $\eps \to 0$ (which is again a consequence of \eqref{eq:proof_th_intersection1}). By dominated convergence theorem, it implies that
\[
\Expect{ \abs{ \Mc_{x_0,Y;\eps}^{D',a}(A) - |\log \eps| \eps^{-a} \int_A \indic{\frac{1}{\eps}L_{x,\eps}^{(0)} \geq 2a |\log \eps|^2} \indic{L_{x,\eps}^{(1)} =0} dx } }
\]
tends to zero as $\eps \to 0$.
Since $\Mc_{x_0,Y;\eps}^{D'a}$ converges in probability towards $\Mc_{x_0,Y}^{D'a}$ (Proposition \ref{prop:def_measures}), this shows that
\[
\abs{\log \eps} \eps^{-a} \indic{\frac{1}{\eps}L_{x,\eps}^{(0)} \geq 2a |\log \eps|^2} \indic{L_{x,\eps}^{(1)} =0} dx
\]
converges in probability to the same limiting measure. Similarly, the second right hand side term of \eqref{eq:proof_P2} converges in probability towards $\Mc_{Y,z}^{D,a}$ which overall yields
\[
\Mc^{D,a}_{x_0,z} = \Mc_{x_0,Y}^{D'a} + \Mc_{Y,z}^{D,a} + \Mc_{(D',x_0,Y),(D,Y,z)}^a.
\]
This is Property \ref{charac2} and it completes the proof.
\end{proof}

The rest of this section is devoted to the uniqueness part of Theorem \ref{th:charac}.

\begin{proof}[Proof of Theorem \ref{th:charac}, uniqueness]
Let
$\left (\mu^{\Dc,a}_{\Xc,\Zc}, \Dc\Xc\Zc \in \Sc \right)$
be a process satisfying Properties \ref{charac1}-\ref{charac3}. Let $D$ be a nice domain, $x_0 \in D$ and $z \in \partial D$ be a nice point. We are going to identify the law of $\mu^{D,a}_{x_0,z}$. As mentioned in Section \ref{subsec:organisation}, the identification of more general marginals follows along the same lines. The only extra difficulty lies in the notations. We start this proof by noticing that we can find a deterministic angle $\theta \in \R$ such that all the lines with angle $\theta$ are not seen by the measure $\mu^{D,a}_{x_0,z}$. Here and in the following, we say that the angle of a line $L$ is $\theta$ if we can write $L = x + e^{i\theta} (\{0\} \times \R)$ for some $x \in \C$.

\begin{lemma}\label{lem:charac3old}
There exists an angle $\theta \in \R$ such that for all $\eps>0$,
\[
\lim_{p \to \infty}
\# \left\{ q \in \Z: \mu^{D,a}_{x_0,z} \left( e^{i \theta} \left( 2^{-p} (q + (0,1]) \times \R \right) \right) \geq \eps \right\} = 0 \quad \quad \mathrm{a.s.}
\]
\end{lemma}

\begin{proof}
We proceed by contradiction. Assume that for all $\theta \in \R$, there exists $\eps_\theta >0$ such that the event $E_\theta$ that
\[
\limsup_{p \to \infty}
\# \left\{ q \in \Z: \mu^{D,a}_{x_0,z} \left( e^{i \theta} \left( 2^{-p} (q + (0,1]) \times \R \right) \right) \geq \eps_\theta \right\} \geq 1
\]
holds with positive probability $p_\theta$. We first argue that on the event $E_\theta$, there exists a line $L_\theta$ with angle $\theta$ such that $\mu^{D,a}_{x_0,z}(L_\theta) \geq \eps_\theta$.
Indeed, on the event $E_\theta$, there exists an increasing sequence of integers $(p_n)_{n \geq 1}$ and a sequence $(q_n)_{n \geq 1} \subset \Z$ such that $\mu^{D,a}_{x_0,z} \left( e^{i \theta} \left( 2^{-p_n} (q_n + (0,1]) \times \R \right) \right) \geq \eps_\theta$ for all $n \geq 1$.
Moreover, because the total mass of $\mu^{D,a}_{x_0,z}$ is almost surely finite, we can extract a subsequence to ensure that $e^{i \theta} \left( 2^{-p_n} (q_n + (0,1]) \times \R \right), n \geq 1,$ is a decreasing sequence of sets. The intersection of those sets is a line $L_\theta$ with angle $\theta$ satisfying the desired property that $\mu^{D,a}_{x_0,z}(L_\theta) \geq \eps_\theta$.

Now, since $[0,\pi)$ is uncountable, there exists $\eta >0$ such that $\{ \theta \in [0,\pi): p_\theta > \eta, \eps_\theta > \eta \}$ is infinite. Let $\{\theta_k, k \geq 1\}$ be a subset of this set. For all $k \geq 1$, we have by the Paley-Zygmund inequality
\begin{align*}
\Prob{\sum_{n \geq 1} \mathbf{1}_{E_{\theta_n}} \geq \frac{\eta}{2} k}
& \geq \Prob{\sum_{1 \leq n \leq k} \mathbf{1}_{E_{\theta_n}} \geq \frac{1}{2} \Expect{\sum_{1 \leq n \leq k} \mathbf{1}_{E_{\theta_n}}} } \\
& \geq \frac{1}{4} \frac{\Expect{\sum_{1 \leq n \leq k} \mathbf{1}_{E_{\theta_n}}}^2}{\Expect{\left(\sum_{1 \leq n \leq k} \mathbf{1}_{E_{\theta_n}} \right)^2} }
\geq \frac{\eta^2}{4}.
\end{align*}
Hence the probability that an infinite number of events $E_{\theta_k}, k \geq 1$, occur is positive. On this event, we have
\[
\sum_{k \geq 1} \mu^{D,a}_{x_0,z}(L_{\theta_k}) \geq \eta \sum_{k \geq 1} \mathbf{1}_{E_{\theta_k}} = \infty.
\]
But because $\mu^{D,a}_{x_0,z}$ is non-atomic (Property \ref{charac3}), we almost surely have
\[
\sum_{k \geq 1} \mu^{D,a}_{x_0,z}(L_{\theta_k}) = \mu^{D,a}_{x_0,z} \left( \bigcup_{k \geq 1} L_{\theta_k} \right) \leq \mu^{D,a}_{x_0,z}(\C)
\]
which is almost surely finite (Property \ref{charac1} implies that it has a finite first moment). We have obtained an absurdity which concludes the proof.
\end{proof}

This result will be used at the very end of the proof; see \eqref{eq:lem_used}. Roughly speaking, in the course of the proof we will decompose the measure into small pieces and Lemma \ref{lem:charac3old} ensures that these pieces are indeed small.

Without loss of generality, we will assume that the specific angle $\theta$ provided by Lemma \ref{lem:charac3old} is equal to 0. In other words, the measure $\mu^{D,a}_{x_0,z}$ almost surely vanishes on all vertical lines. We will also assume for convenience that $D \subset (0,1) \times \R$.

Let us introduce some notations. 
We will need to consider small portions of the domain which are well-separated from one another. For this reason, we introduce a Cantor-type set $K^\infty$ which we define now. Let $p_0 \geq 1$ (to be thought of as large) and for all $n \geq 1$, let $\mathfrak{D}_n$ be the set of dyadic points of generation exactly $n$, i.e.
\[
\mathfrak{D}_n = \Big\{ (2m+1) 2^{-n} : m \in \{0, \dots, 2^{n-1} -1 \} \Big\}.
\]
For instance, $\mathfrak{D}_1 = \{1/2\}$, $\mathfrak{D}_2 = \{1/4, 3/4\}$, $\mathfrak{D}_3 = \{1/8, 3/8, 5/8, 7/8\}$, etc. We now define $K^0 = [0,1]$ and for all $n \geq 1$,
\begin{equation}
\label{eq:K_n}
K^n := K^{n-1} \setminus \bigcup_{x \in \mathfrak{D}_n} (x - 2^{-(p_0 + 2n)}, x + 2^{-(p_0+2n)} ). 
\end{equation}
We then define
\begin{equation}
\label{eq:K_infty}
K^\infty := \bigcap_{n \geq 0} K^n.
\end{equation}
Later in the proof, we will restrict some measures to the set $D \cap K^\infty \times \R$. This will capture almost entirely our measures since the Lebesgue measure of $D \backslash \left( D \cap K^\infty \times \R \right)$ is at most $C 2^{-p_0}.$ Note also that, as $p_0 \to \infty$, $K^\infty$ increases to $[0,1] \setminus \bigcup_{n \geq 1} \mathfrak{D}_n$.

\medskip

We now start more concretely the proof of Theorem \ref{th:charac}.
Let $p \geq 1$ and $(B_t, t \leq \tau_{\partial D})$ be a Brownian motion distributed according to $\prob^D_{x_0,z}$.
We are going to keep track of the successive Brownian hitting points of $2^{-p} \Z \times \R$: define $\sigma_0^p := 0$, $x^p_0 := x_0$ and $D_0^p := D \cap \left( 2^{-p} \floor{2^p x_0} + \left(-2^{-p}, 2^{-p} \right) \times \R \right)$ and for all $i \geq 1$,
\begin{equation}
\label{eq:proof_thm_charac}
\sigma_i^p := \inf \{ t > \sigma_{i-1}^p: B_t \notin D_{i-1}^p \}, x_i^p := B_{\sigma_i^p}
\mathrm{~and~}
D_i^p := D \cap \left( x_i^p + \left(-2^{-p}, 2^{-p} \right) \times \R \right).
\end{equation}
Let $I_p := \sup \{i \geq 1: \sigma_i^p \leq \tau_{\partial D} \}$. Note that $\sigma_{I_p}^p = \tau_{\partial D}$ and $x^p_{I_p} = z$. See Figure \ref{fig1} for an illustration of these notations.
Let
\[
\Dc^p\Xc^p\Zc^p := \{ (D^p_i,x^p_i,x^p_{i+1}), i=0 \dots I_p -1 \}
\]
and let
\[
\left( \bar{\mu}_{\Xc,\Zc}^{\Dc,a}, \Dc \Xc \Zc \subset \Dc^p\Xc^p\Zc^p \right)
\]
be the process so that conditionally on $x_i^p, i = 1 \dots I_p-1$, it has the same law as
\[
\left( \mu^{\Dc,a}_{\Xc,\Zc}, \Dc \Xc \Zc \subset \Dc^p\Xc^p\Zc^p \right).
\]
Note that with the definition \eqref{eq:proof_thm_charac}, $D_i^p$ may be formed of several connected components. To be more precise, we define $D_i^p$ as being the connected component that contains $x_i^p$ which is a nice domain belonging to $\Dc$. $x_{i+1}^p$ being almost surely a nice boundary point of $D_i^p$ and the $x_i^p, i \geq 1$ being almost surely pairwise distinct, the above random measures are well defined.
An elementary iteration of Property \ref{charac2} shows that
\begin{equation}
\label{eq:decomposition}
\bar{\mu}_{x_0,z}^{D,a}
:=
\sum_{\Dc\Xc\Zc \subset \Dc^p\Xc^p\Zc^p} \bar{\mu}_{\Xc,\Zc}^{\Dc,a}.
\end{equation}
has the same law as $\mu^{D,a}_{x_0,z}$. These definitions are consistent and by Kolmogorov's extension theorem, we can define $x^p_i, i =0 \dots I_p$, $p \geq 1$, $\bar{\mu}_{\Xc,\Zc}^{\Dc,a}$, $\Dc\Xc\Zc \subset \cup_{p \geq 1} \Dc^p\Xc^p\Zc^p$ on the same probability space.

In the rest of the proof, we will work on the specific probability space given by Kolmogorov's extension theorem as above. We will drop the bar and simply write
\[
\mu^{D,a}_{x_0,z},
~\mu^{\Dc,a}_{\Xc,\Zc}
\mathrm{~instead~of~}
\bar{\mu}^{D,a}_{x_0,z},
~\bar{\mu}^{\Dc,a}_{\Xc,\Zc}.
\]
In the following, we will denote by $\Fc_p$ (resp. $\Fc_\infty$) the $\sigma$-algebra generated by $x_i^p, i=1 \dots I_p-1$ (resp. $x_i^p, i=1 \dots I_p-1, p \geq 1$) and
\begin{equation}
\label{eq:charac_def_mup}
\mu_p(dx) = \Expect{ \left. \mu^{D,a}_{x_0,z}(dx) \right\vert \Fc_p}.
\end{equation}
By \eqref{eq:decomposition} and Property \ref{charac1}, $\mu_p(dx)$ does not depend on the process
$(\mu^{\Dc,a}_{\Xc,\Zc}, \Dc\Xc\Zc \in \Sc)$
any more since it is equal to
\begin{equation}
\label{eq:charac_def_mup2}
\sum_{r=1}^{I_p} \sum_{\{i_1 \dots i_r\} \subset \{0 \dots I_p-1\} } \int_{\mathsf{a} \in E(a,r)} d \mathsf{a} \prod_{k=1}^r \psi_{x_{i_k}^p,x_{i_k+1}^p}^{D^p_{i_k},a_k}(x) dx.
\end{equation}

The following lemma is a key feature of the proof:

\begin{lemma}\label{lem:charac}
There exists an a.s. finite random Borel measure $\mu_\infty$ such that for all bounded measurable function $f: D \to \R$, $(\scalar{\mu_p,f}, p \geq 1)$ is a martingale and converges a.s. to $\scalar{ \mu_\infty,f}$.
\end{lemma}

\begin{proof}
Let $\mathcal{P}$ be a countable $\pi$-system generating the Borel sets of $D$. For all $A \in \mathcal{P}$, $(\mu_p(A), \Fc_p)_{p \geq 1}$ is a non-negative martingale thanks to \eqref{eq:charac_def_mup}. Hence, almost surely for all $A \in \mathcal{P}$, $\mu_p(A)$ converges towards some $L(A)$. By standard arguments (see Section 6 of \cite{berestycki2017} for instance), one can show that it implies that there exists an a.s. finite random Borel measure $\mu_\infty$ such that almost surely for all $A \in \mathcal{P}$, $L(A) = \mu_\infty(A)$. It moreover implies that almost surely for all bounded measurable function $f$, $\scalar{\mu_p,f}$ converges towards $\scalar{\mu_\infty,f}$.
\end{proof}

Since $\mu_\infty$ is entirely characterised by Properties \ref{charac1}-\ref{charac3}, it is enough to show that $\mu^{D,a}_{x_0,z} = \mu_\infty$ a.s. to conclude the proof of Theorem \ref{th:charac}. Since two finite measures which coincide on a (countable) $\pi$-system generating the Borel sets of $\C$ are equal, it is further enough to show that for all Borel set $A \subset \C$, $\mu^{D,a}_{x_0,z}(A) = \mu_\infty(A)$ a.s. We then notice that it is enough to show that for all $t >0$ and Borel set $A$,
\begin{equation}
\label{eq:proof_charac_goal}
\Expect{ \left. e^{-t\mu^{D,a}_{x_0,z}(A)} \right\vert \Fc_\infty } = e^{-t\mu_\infty(A)} \quad \quad \mathrm{a.s.}
\end{equation}
Indeed, it proves that conditionally on  $\Fc_\infty$ the Laplace transform of $\mu^{D,a}_{x_0,z}(A)$ is almost surely equal to the Laplace transform of the constant $\mu_\infty(A)$ on all the positive rational numbers which in turn proves that $\mu^{D,a}_{x_0,z}(A) = \mu_\infty(A)$ a.s.
Until the end of the proof we will fix such a Borel set $A$.
We reduce the problem one last time: recall the definition \eqref{eq:K_infty} of the Cantor-type set $K^\infty$ (which depends on the integer $p_0$) that we introduced at the beginning of the proof and recall that $K^\infty$ increases with $p_0$ towards $[0,1] \setminus \bigcup_{n \geq 1} \mathfrak{D}_n$ (see the discussion below \eqref{eq:K_infty}). By computing the first moment of the variables below, we see that
\[
\mu^{D,a}_{x_0,z} (\bigcup_{n \geq 1} \mathfrak{D}_n \times \R) = \mu_\infty (\bigcup_{n \geq 1} \mathfrak{D}_n \times \R) = 0 \quad \quad \text{a.s.}
\]
Therefore, as $p_0 \to \infty$,
\[
\mu^{D,a}_{x_0,z}(A \cap K^\infty) \to \mu^{D,a}_{x_0,z}(A) \quad \text{and} \quad
\mu_\infty (A \cap K^\infty) \to \mu_\infty(A) \quad \quad \text{a.s.}
\]
In other words, we can safely assume that $A$ is included in $K^\infty$. This assumption will be made for the rest of the proof.

Our objective is to show \eqref{eq:proof_charac_goal}. Without loss of generality, we can assume that $t=1$.
One direction is easy: by \eqref{eq:charac_def_mup}, we have
\[
\Expect{ \left. \mu^{D,a}_{x_0,z}(A) \right\vert \Fc_p} = \mu_p(A) \quad \quad \mathrm{a.s.},
\]
so by Jensen's inequality,
\[
\Expect{\left. e^{-\mu^{D,a}_{x_0,z}(A)} \right\vert \Fc_p} \geq e^{-\mu_p(A)} \quad \quad \mathrm{a.s}
\]
By Lemma \ref{lem:charac}, $\mu_p(A) \to \mu_\infty(A)$ a.s. So by letting $p \to \infty$ we get
\[
\Expect{\left. e^{-\mu^{D,a}_{x_0,z}(A)} \right\vert \Fc_\infty} \geq e^{-\mu_\infty(A)} \quad \quad \mathrm{a.s}
\]
For the reverse direction, we use Lemma 3.12 of \cite{BiskupLouidor} which provides a ``reverse Jensen'' inequality that we recall.

\begin{lemmaa}[\cite{BiskupLouidor}, Lemma 3.12]\label{lem:reverse_Jensen}
If $X_1, \dots, X_n$ are non-negative independent random variables, then for each $\eps >0$,
\[
\Expect{ e^{- \sum_{i=1}^n X_i } } \leq \exp \left( -e^{-\eps} \sum_{i=1}^n \Expect{X_i;X_i \leq \eps} \right).
\]
\end{lemmaa}

Let $p \geq 1$ be much larger than $p_0$ and let $n \geq 1$ be such that $p_0 + 2n = p$ (or such that $p_0 + 2n = p-1$, depending on the parity).
Recall the definition \eqref{eq:K_n} of $K^n$. We will denote by $K^{n,m}, m =1, \dots, 2^n$, the connected components of $K^n$.
We notice that conditioned on $\Fc_p$, the measures $\mu^{D,a}_{x_0,z}(\bigcdot \cap K^{n,m})$, $m =1 \dots 2^n$, are independent. Indeed, looking at \eqref{eq:decomposition} we see that Property \ref{charac2bis} implies that conditioned on $\Fc_p$, $\mu^{D,a}_{x_0,z}(\bigcdot \cap A_1)$ and $\mu^{D,a}_{x_0,z}(\bigcdot \cap A_2)$ are independent as soon as the projections of $A_1$ and $A_2$ on the real axis are at distance at least $2\times 2^{-p}$ from each other. 
The whole introduction of the set $K^\infty$ is motivated by this fact.
Now, because $A \subset K^\infty$ and by Lemma \ref{lem:reverse_Jensen}, we deduce that for each $\eps >0$,
\begin{align*}
& \Expect{ \left. e^{- \mu^{D,a}_{x_0,z}(A) } \right\vert \Fc_p} \\
& \leq \exp \left( -e^{-\eps} \sum_{m=1}^{2^n} \Expect{ \left. \mu^{D,a}_{x_0,z}(A \cap K^{n,m}) ; \mu^{D,a}_{x_0,z}(A \cap K^{n,m}) \leq \eps \right\vert \Fc_p } \right) \quad \quad \mathrm{a.s.}
\end{align*}
To conclude that
\[
\Expect{ \left. e^{- \mu^{D,a}_{x_0,z}(A) } \right\vert \Fc_\infty } \leq
e^{ - \mu_\infty(A) } \quad \quad \mathrm{a.s.,}
\]
it is thus enough to show that a.s.
\begin{equation*}
\liminf_{\eps \to 0} \liminf_{p \to \infty} \sum_{m=1}^{2^n} \Expect{ \left. \mu^{D,a}_{x_0,z}(A \cap K^{n,m}) ; \mu^{D,a}_{x_0,z}(A \cap K^{n,m}) \leq \eps \right\vert \Fc_p }
\geq \mu_\infty(A).
\end{equation*}
We have
\begin{align*}
& \liminf_{\eps \to 0} \liminf_{p \to \infty} \sum_{m=1}^{2^n} \Expect{ \left. \mu^{D,a}_{x_0,z}(A \cap K^{n,m}) ; \mu^{D,a}_{x_0,z}(A \cap K^{n,m}) \leq \eps \right\vert \Fc_p } \\
& \geq \mu_\infty(A) - \limsup_{\eps \to 0} \limsup_{p \to \infty} \sum_{m=1}^{2^n} \Expect{ \left. \mu^{D,a}_{x_0,z}(A \cap K^{n,m}) ; \mu^{D,a}_{x_0,z}(A \cap K^{n,m}) > \eps \right\vert \Fc_p } \quad \quad \mathrm{a.s.}
\end{align*}
But by Lemma \ref{lem:charac3old} and dominated convergence theorem,
\begin{align}
\label{eq:lem_used}
& \Expect{ \sum_{m=1}^{2^n} \Expect{ \left. \mu^{D,a}_{x_0,z}(A \cap K^{n,m}) ; \mu^{D,a}_{x_0,z}(A \cap K^{n,m}) > \eps \right\vert \Fc_p } } \\
\nonumber
& =
\sum_{m=1}^{2^n} \Expect{ \mu^{D,a}_{x_0,z} (A \cap K^{n,m}) ; \mu^{D,a}_{x_0,z} (A \cap K^{n,m}) > \eps }
\end{align}
tends to zero as $p \to \infty$ (recall that $n \to \infty$ as $p \to \infty$). Hence, by extracting a subsequence if necessary, we have
\[
\limsup_{\eps \to 0} \limsup_{p \to \infty} \sum_{m=1}^{2^n} \Expect{ \left. \mu^{D,a}_{x_0,z}(A \cap K^{n,m}) ; \mu^{D,a}_{x_0,z}(A \cap K^{n,m}) > \eps \right\vert \Fc_p } = 0 \quad \quad \mathrm{a.s.}
\]
which concludes the proof of Theorem \ref{th:charac}.
\end{proof}

\section{Application to random walk: proof of Theorem \ref{th:convergence}}\label{sec:RW}

We start off by defining the multipoint analogue of $\mu_{x_0,z;N}^{U,a}$. Let $r \geq 1$ and $\Dc\Xc\Zc = \{(D^i,x_i,z_i)$, $i =1 \dots r \} \in \Sc$. Let $X^{(i)}$, $i=1 \dots r$, be $r$ independent random walk distributed according to $\prob_{Nx_i,Nz_i}^{D^i_N}$ or according to $\prob_{Nx_i}^{D^i_N}$ and let $\ell_x^{(i)}$ be their associated local times. We define simultaneously for all $\Dc'\Xc'\Zc' = \{(D^i,x_i,z_i), i \in I\} \subset \Dc\Xc\Zc$ the measures given by: for all Borel set $A$,
\begin{equation}
\label{eq:def_multipoint_discrete}
\mu_{\Xc',\Zc';N}^{\Dc',a}(A) := \frac{\log N}{N^{2 - a}} \sum_{x \in \Z^2} \indic{x/N \in A} \indic{ \sum_{i \in I} \ell_x^{(i)} \geq g a \log^2 N} \indic{ \forall i \in I, \ell_x^{(i)} >0 }
\end{equation}
under the probability
$\bigotimes_{i=1}^r \prob_{Nx_i,Nz_i}^{D^i_N}$. We define similarly the unconditioned measures $\mu_{\Xc';N}^{\Dc',a}$, $\Dc'\Xc' \subset \Dc\Xc$, under $\bigotimes_{i=1}^r \prob_{Nx_i}^{D^i_N}$.

\subsection{Tightness and first moment estimates}\label{sec:tight}

In this section we fix a nice domain $D$.
We start by recalling Green's function and Poisson kernel asymptotic behaviours. Recall the notations of Section \ref{subsec:notations}.

\begin{lemma}[Green's function]\label{lem:Green}
Let $K \Subset D$. There exist $C,C_K>0$ such that for all $x,y \in \Z^2$,
\begin{gather}
\label{eq:lem_Green_upper_bound}
G^{D_N}(x,y) \leq g \log \frac{N}{\abs{x-y} \vee 1} + C, \mathrm{~if~} x,y \in D_N, \\
\label{eq:lem_Green_lower_bound}
G^{D_N}(x,y) \geq g \log \frac{N}{\abs{x-y} \vee 1} - C_K, \mathrm{~if~} \frac{x}{N}, \frac{y}{N} \in K.
\end{gather}
Moreover, for all $x\neq y \in D$, we have
\begin{gather}
\label{eq:lem_Green_limit_on_diagonal}
\lim_{N \to \infty} G^{D_N} \left( \floor{Nx} , \floor{Nx} \right) - g \log N = g \log \CR(x,D)+c_0, \\
\label{eq:lem_Green_limit_off_diagonal}
\lim_{N \to \infty} G^{D_N} \left( \floor{Nx} , \floor{Ny} \right) = g G^D(x,y),
\end{gather}
where $c_0$ is the universal constant defined in \eqref{eq:c0}.
\end{lemma}

\begin{proof}
\eqref{eq:lem_Green_upper_bound} and \eqref{eq:lem_Green_lower_bound} are direct consequences of \cite{lawler1996intersections} Theorem 1.6.2 and Proposition 1.6.3. \eqref{eq:lem_Green_limit_on_diagonal} and \eqref{eq:lem_Green_limit_off_diagonal} are contained in Theorem 1.17 of \cite{BiskupLectures}.
\end{proof}

\begin{lemma}[Poisson kernel]\label{lem:Poisson}
Let $K \Subset D$ and $\alpha >0$. For all $N$ large enough, $x,y \in K$ and $z \in \partial D$ a nice point, we have
\begin{equation}
\label{eq:lem_Poisson}
\abs{ \frac{H^{D_N}(\floor{Nx},\floor{Nz})}{H^{D_N}(\floor{Ny},\floor{Nz})} - \frac{H^D(x,z)}{H^D(y,z)} } \leq \alpha.
\end{equation}
Moreover, for all $x \in D$, the following weak convergence holds:
\begin{equation}
\label{eq:lem_Poisson_convergence}
\sum_{z \in \partial D_N} H^{D_N}( \floor{Nx}, z) \delta_{z/N}(\cdot) \xrightarrow[N \to \infty]{\mathrm{weakly}} \int_{\partial D} H^D(x,z) \delta_z(\cdot).
\end{equation}
\end{lemma}

\begin{proof}
Statements of the flavour of \eqref{eq:lem_Poisson} have been extensively studied to show the convergence of loop-erased random walk towards $\mathrm{SLE}_2$. \eqref{eq:lem_Poisson} is a direct consequence of \cite[Lemma 1.2]{yadin2011} for instance.
\eqref{eq:lem_Poisson_convergence} is the content of \cite[Lemma 1.23]{BiskupLectures}.
\end{proof}

These two lemmas allow us to derive the first moment estimates that we need. In the following we let $\Dc\Xc\Zc = \{ (D^i,x_i,z_i), i=1 \dots r\} \in \Sc$ and $\Dc\Xc = \{(D^i,x_i), i = 1 \dots r\}$ and we denote $\ell_x^{(i)}$ the local times associated to the $i$-th random walk as at the very beginning of Section \ref{sec:RW}. For all nice domain $D$ and $x_0 \in D$, we will also denote for all $x \in \C$,
\[
\varphi_{x_0}^{D,a}(x) = G^D(x_0,x) \CR(x,D)^a \indic{x \in D}.
\]

\begin{lemma}\label{lem:1-point}
There exists $C>0$ such that for all $N \geq 1$ and $x \in \C$,
\begin{align}
\label{eq:lem_bound}
& \log (N) N^a \bigotimes_{i=1}^r \prob^{D^i_N}_{Nx_i} \left( \sum_{i=1}^r \ell_{\floor{Nx}}^{(i)} \geq ga \log^2 N, \forall i=1 \dots r, \ell_{\floor{Nx}}^{(i)} >0 \right) \\
& \leq C \prod_{i=1}^r \abs{ \log \left( \frac{\abs{x-x_i}}{C} \vee \frac{1}{N} \right)}. \nonumber
\end{align}
Let $K \Subset \cap_{i=0}^{r-1} D^i$. There exists $C>0$ depending on $K$, such that for all $N$ large enough and $x \in K$,
\begin{align}
\label{eq:lem_bound2}
& \log (N) N^a
\bigotimes_{i=1}^r \prob^{D^i_N}_{Nx_i,Nz_i} \left( \sum_{i=1}^r \ell_{\floor{Nx}}^{(i)} \geq ga \log^2 N, \forall i=1 \dots r, \ell_{\floor{Nx}}^{(i)}>0 \right) \\
& \leq
C \prod_{i=1}^r \abs{ \log \left( \frac{\abs{x-x_i}}{C} \vee \frac{1}{N} \right)}. \nonumber
\end{align}
Moreover, for all $x \in \C$,
\begin{align}
\label{eq:lem_limit_unconditioned}
& \lim_{N \to \infty} \log (N) N^a \bigotimes_{i=1}^r \prob^{D^i_N}_{Nx_i} \left( \sum_{i=1}^r \ell_{\floor{Nx}}^{(i)} \geq ga \log^2 N, \forall i=1 \dots r, \ell_{\floor{Nx}}^{(i)}>0 \right) \\
& = e^{ \frac{c_0 a}{g} } 
\int_{\mathsf{a} \in E(a,r)} d \mathsf{a}
\prod_{k=1}^r \varphi_{x_k}^{D^k,a_k}(x) \nonumber
\end{align}
and
\begin{align}
\label{eq:lem_limit}
& \lim_{N \to \infty} \log (N) N^a \bigotimes_{i=1}^r \prob^{D^i_N}_{Nx_i,Nz_i} \left( \sum_{i=1}^r \ell_{\floor{Nx}}^{(i)} \geq ga \log^2 N, \forall i=1 \dots r, \ell_{\floor{Nx}}^{(i)}>0 \right) \\
& = e^{ \frac{c_0 a}{g} }
\int_{\mathsf{a} \in E(a,r)} d \mathsf{a}
\prod_{k=1}^r \psi_{x_k,z_k}^{D^k,a_k}(x) \nonumber
\end{align}
where $E(a,r)$ is the $(n-1)$-dimensional simplex defined in \eqref{eq:def_simplex}.
\end{lemma}

\begin{proof}[Proof of Lemma \ref{lem:1-point}]
We start by proving \eqref{eq:lem_bound2} and \eqref{eq:lem_limit}. To ease notations, we will write
\[
\prob := \bigotimes_{i=1}^r \prob_{Nx_i,Nz_i}^{D^i_N}.
\]
Let $x \in \Z^2$. We have
\begin{align}
\label{eq:proof_1_point1}
& \prob \left( \sum_{i=1}^r \ell_x^{(i)} \geq ga \log^2 N, \forall i=1 \dots r, \ell_x^{(i)}>0 \right) \\
& = \prod_{i=1}^r \prob_{Nx_i,Nz_i}^{D^i_N} \left( \ell_x^{(i)} >0 \right)
\prob \left( \left. \sum_{i=1}^r \ell_x^{(i)} \geq ga \log^2 N \right\vert \forall i=1 \dots r, \ell_x^{(i)}>0 \right). \nonumber
\end{align}
The Markov property gives that for all $i=1 \dots r$,
\begin{align*}
\prob_{Nx_i,Nz_i}^{D^i_N} \left( \ell_x^{(i)} >0 \right) & = \prob^{D^i_N}_{Nx_i} \left( \ell_x^{(i)} > 0 \right) \frac{\prob^{D^i_N}_x \left( X^{(i)}_{\tau_{\partial D^i_N}} = N z_i \right)}{\prob^{D^i_N}_{Nx_i} \left( X^{(i)}_{\tau_{\partial D^i_N}} = N z_i \right)} \\
& = \frac{G^{D^i_N}(Nx_i,x)}{G^{D^i_N}(x,x)} \frac{H^{D^i_N}(x,Nz_i)}{H^{D^i_N}(Nx_i,Nz_i)}.
\end{align*}
Moreover, under $\prob^{D^i_N}_x$, $\ell_x^{\tau_{\partial D^i_N}}$ is an exponential variable with mean $G^{D^i_N}(x,x)$ which is independent of $X^{(i)}_{\tau_{\partial D^i_N}}$ (see Lemma \ref{lem:app_indep_local_time}). Therefore, conditioning on $X^{(i)}_{\tau_{\partial D^i_N}}$ does not change the law of $\ell_x^{\tau_{\partial D^i_N}}$ and
\begin{align}
\label{eq:proof_lem_1point2}
& \prob \left( \left. \sum_{i=1}^r \ell_x^{(i)} \geq ga \log^2 N \right\vert \forall i=1 \dots r, \ell_x^{(i)}>0 \right) \\
& = \int_{[0,\infty)^r} dt_1 \dots dt_r e^{-\sum_{i=1}^r t_i} \indic{ \sum_{i=1}^r G^{D_N^i}(x,x) t_i \geq ga \log^2 N}. \nonumber
\end{align}
To bound this term from above, we use \eqref{eq:lem_Green_upper_bound} which allows us to bound
\[
\indic{ \sum_{i=1}^r G^{D_N^i}(x,x) t_i \geq ga \log^2 N} \leq \indic{ (g\log N + C) \sum_{i=1}^r t_i \geq ga \log^2 N}
\]
which yields
\begin{equation}
\label{eq:proof_lem_1point}
\prob \left( \left. \sum_{i=1}^r \ell_x^{(i)} \geq ga \log^2 N \right\vert \forall i=1 \dots r, \ell_x^{(i)}>0 \right) \leq C (\log N)^{r-1} N^{-a}.
\end{equation}
\eqref{eq:lem_Green_upper_bound}, \eqref{eq:lem_Green_lower_bound} and \eqref{eq:lem_Poisson} then concludes the proof of \eqref{eq:lem_bound2}. To get \eqref{eq:lem_limit}, we come back to \eqref{eq:proof_lem_1point2} which gives
\begin{align*}
& \prob \left( \left. \sum_{i=1}^r \ell_x^{(i)} \geq ga \log^2 N \right\vert \forall i=1 \dots r, \ell_x^{(i)}>0 \right) \\
& = \prob \left( \left. \sum_{i=1}^{r-1} \ell_x^{(i)} \geq ga \log^2 N \right\vert \forall i=1 \dots r-1, \ell_x^{(i)}>0 \right)
+ e^{-ga \log^2N/G^{D^r_N}(x,x)}
\\
& \times \int_{[0,\infty)^{r-1}} dt_1 \dots dt_{r-1} \exp \left( \sum_{i=1}^{r-1} \left( \frac{G^{D^i_N}(x,x)}{G^{D^r_N}(x,x)} - 1 \right) \ t_i \right)
\indic{ \sum_{i=1}^{r-1} G^{D^i_N}(x,x) t_i \leq ga\log^2N }.
\end{align*}
\eqref{eq:proof_lem_1point} shows that the first right hand side term is at most $C(\log N)^{r-2} N^{-a}$ which is going to be of smaller order than the second term. Using \eqref{eq:lem_Green_limit_on_diagonal} and performing the change of variable $s_i = t_i/ \log N$ shows that when $x = \floor{Ny}$ the second right hand side term is asymptotically equivalent to
\begin{align*}
e^{\frac{c_0 a}{g}} (\log N)^{r-1} N^{-a} \CR(y,D^r)^a \int_{[0,\infty)^{r-1}} ds_1 \dots ds_{r-1} \prod_{i=1}^{r-1} \left( \frac{\CR(y,D^i)}{\CR(y,D^r)} \right)^{s_i} \indic{\sum_{i=1}^{r-1} s_i \leq a}.
\end{align*}
Using \eqref{eq:lem_Poisson}, this shows that
\begin{align*}
& \lim_{N \to \infty} \log(N) N^a \prob \left( \sum_{i=1}^r \ell_{\floor{Ny}}^{(i)} \geq ga \log^2 N, \forall i=1 \dots r, \ell_{\floor{Ny}}^{(i)}>0 \right)
= \prod_{i=1}^r G^{D^i}(x_i,y) \frac{H^{D^i}(y,z_i)}{H^{D^i}(x_i,z_i)} \\
& \times e^{\frac{c_0 a}{g}} \int_{ [0,a]^{r-1} } d s_1 \dots ds_{r-1} \prod_{i=1}^{r-1} \CR(y,D^i)^{s_i} \CR(y,D^r)^{a - \sum_{i=1}^{r-1} s_i} \indic{ \sum_{i=1}^{r-1} s_i < a }
\end{align*}
which proves \eqref{eq:lem_limit}.

We omit the proofs of \eqref{eq:lem_bound} and \eqref{eq:lem_limit_unconditioned} which are very similar and even slightly easier since there is no conditioning to deal with. We nevertheless mention that in \eqref{eq:lem_bound}, we do not need to restrict ourselves to the bulk of the domains (compared to \eqref{eq:lem_bound2}) because the probability increases with the domains. We can thus assume that all the points we consider are deep inside the domains. This finishes the proof.
\end{proof}

We are now ready to prove:

\begin{proposition}[Tightness]\label{prop:tightness}
The sequences
\[
\left( (\mu^{\Dc',a}_{\Xc';N}, \Dc'\Xc' \subset \Dc\Xc), N \geq 1 \right)
\mathrm{~and~}
\left( (\mu^{\Dc',a}_{\Xc',\Zc';N}, \Dc'\Xc'\Zc' \subset \Dc\Xc\Zc), N \geq 1 \right)
\]
are tight for the product topology of, respectively, weak and vague convergence on $\bigcap_{D \in \Dc'} D, \Dc' \subset \Dc$. Moreover, for any Borel set $A \subset \C$,
\begin{align}
\label{eq:prop_limit_first_moment_unconditioned}
& \lim_{N \to \infty} \Expect{ \mu^{\Dc,a}_{\Xc;N}(A)}
= e^{ \frac{c_0 a}{g} } \int_A dx 
\int_{\mathsf{a} \in E(a,r)} d \mathsf{a}
\prod_{k=1}^r \varphi_{x_k}^{D^k,a_k}(x)
\end{align}
and if $A$ is compactly included in $\cap_{i=1}^r D^i$,
\begin{align}
\label{eq:prop_limit_first_moment}
& \lim_{N \to \infty} \Expect{ \mu^{\Dc,a}_{\Xc,\Zc;N}(A)}
= e^{ \frac{c_0 a}{g} } \int_A dx 
\int_{\mathsf{a} \in E(a,r)} d \mathsf{a}
\prod_{k=1}^r \psi_{x_k,z_k}^{D^k,a_k}(x)
\end{align}
where $E(a,r)$ is the $(r-1)$-dimensional simplex defined in \eqref{eq:def_simplex}.
\end{proposition}

\begin{proof}[Proof of Proposition \ref{prop:tightness}]
To prove the desired tightness, it is enough to show that for all $\Dc'\Xc' \subset \Dc\Xc$ and $\Dc'\Xc'\Zc' \subset \Dc\Xc\Zc$ and $K \Subset \bigcap_{D \in \Dc'} D$, the sequences of real-valued random variables
\[
\left( \mu_{\Xc';N}^{\Dc',a}(\C), N \geq 1 \right)
\mathrm{~and~}
\left( \mu_{\Xc',\Zc';N}^{\Dc',a}(K), N \geq 1 \right)
\]
are tight. This is a direct consequence of Lemma \ref{lem:1-point}: \eqref{eq:lem_bound} and \eqref{eq:lem_bound2} show that
\[
\Expect{\mu_{\Xc';N}^{\Dc',a}(\C)}
\mathrm{~and~}
\Expect{\mu_{\Xc',\Zc';N}^{\Dc',a}(K)}
\]
are uniformly bounded in $N$.
\eqref{eq:prop_limit_first_moment_unconditioned} and \eqref{eq:prop_limit_first_moment} follow from dominated convergence theorem and \eqref{eq:lem_limit_unconditioned} and \eqref{eq:lem_limit} respectively.
\end{proof}

\subsection{Study of the subsequential limits}\label{sec:identification}

As described in Section \ref{subsec:organisation}, we start by showing that we can extract a subsequence such that the convergence holds for all domains and starting/stopping points at the same time.
The difficulty lies in the fact that we consider uncountably many sequences.

\begin{lemma}\label{lem:uncountable_extraction}
Let $(N_k, k \geq 1)$ be an increasing sequence of integers. There exists a subsequence $(N'_k, k \geq 1)$ of $(N_k, k \geq 1)$ such that for all $\Dc'\Xc'\Zc' \in \Sc$,
\[
(\mu_{\Xc,\Zc;N_k'}^{\Dc,a}, \Dc\Xc\Zc \subset \Dc'\Xc'\Zc')
\]
converges as $k \to \infty$ in distribution, relative to the product topology of vague convergence on $\bigcap_{D \in \Dc} D$, $\Dc \subset \Dc'$.
\end{lemma}

Before proving this result, we state an elementary lemma for ease of reference:

\begin{lemma}\label{lem:elementary}
Let $(X_k, k \geq 1)$ be a sequence of random variables. Assume that for all $k \geq 1$ and $p \geq 1$, $X_k$ can be written as
$X_k = Y_{k,p} + Z_{k,p}$
where $Y_{k,p}$ and $Z_{k,p}$ are two non-negative random variables defined on the same probability space.
Assume further that for all $\lambda>0$, 
\[
\lim_{k \to \infty} \Expect{e^{-\lambda Y_{k,p}}}
\mathrm{~and~}
\lim_{p \to \infty} \lim_{k \to \infty} \Expect{e^{-\lambda Y_{k,p}}}
\]
exist and that for all $p \geq 1$, $\sup_{k \geq 1} \Expect{Y_{k,p}} < \infty$ and $\sup_{k \geq 1} \Expect{Z_{k,p}} \to 0$ when $p \to \infty$. Then $(X_k,k \geq 1)$ converges in distribution.
\end{lemma}

\begin{proof}[Proof of Lemma \ref{lem:elementary}]
As $\sup_{k \geq 1} \Expect{X_k} < \infty$, $(X_k, k \geq 1)$ is tight. To show that it converges, it is thus enough to show the pointwise convergence of the Laplace transform.
Take $\lambda >0$.
Since $Z_{k,p}$ is non-negative,
\[
\Expect{e^{-\lambda X_k}} \leq \Expect{e^{-\lambda Y_{k,p}}}
\]
and
\[
\limsup_{k \to \infty} \Expect{e^{-\lambda X_k}} \leq \lim_{k \to \infty} \Expect{e^{-\lambda Y_{k,p}}} \xrightarrow[p \to \infty]{} \lim_{p \to \infty} \lim_{k \to \infty} \Expect{e^{-\lambda Y_{k,p}}}.
\]
On the other hand,
\[
\Expect{e^{-\lambda X_k}} - \Expect{e^{-\lambda Y_{k,p}}} = - \Expect{e^{-\lambda Y_{k,p}} \left( 1 - e^{-\lambda Z_{k,p}} \right)}
\geq - \lambda \Expect{Z_{k,p}}
\]
and
\[
\liminf_{k \to \infty} \Expect{e^{-\lambda X_k}} \geq \lim_{k \to \infty} \Expect{e^{-\lambda Y_{k,p}}} - \lambda \sup_{k \geq 1} \Expect{Z_{k,p}} \xrightarrow[p \to \infty]{} \lim_{p \to \infty} \lim_{k \to \infty} \Expect{e^{-\lambda Y_{k,p}}}.
\]
We have shown that $\Expect{e^{-\lambda X_k}}, k \geq 1,$ converges to $\lim_{p \to \infty} \lim_{k \to \infty} \Expect{e^{-\lambda Y_{k,p}}}$ which concludes the proof.
\end{proof}

\begin{proof}[Proof of Lemma \ref{lem:uncountable_extraction}]
In this proof, the topologies associated to the unconditioned (resp. conditioned) measures will be the topology of weak convergence (resp. vague convergence) on the underlying domain. We will denote by $\mathfrak{D}$ the collection of simply connected domains that can be written as a finite union of discs with rational centres and radii and
\[
\Sc' := \bigcup_{r \geq 1} \left\{ \{ (D_i,x_i), i=1 \dots r \}: \forall i=1 \dots r, D_i \in \mathfrak{D}, x_i \in D_i \cap \Q^2 \right\}.
\]
Notice that $\Sc'$ is countable.

Let $\Dc\Xc \in \Sc$.
By Proposition \ref{prop:tightness}, the sequence $(\mu^{\Dc,a}_{\Xc;N_k}, k \geq 1)$ is tight.
Denote by $(X^{D}_{x;N_k}(t), 0 \leq t \leq \tau^{N_k}_{D,x}), (D,x) \in \Dc\Xc$, the associated random walks, i.e. independent trajectories sampled according to $\P^{D_{N_k}}_x$. The sequence $\left( N_k^{-1} X^D_{x;N_k}(N_k^2 t), t \leq N_k^{-2} \tau^{N_k}_{D,x} \right)_{(D,x) \in \Dc\Xc}, k \geq 1,$ is also tight since it converges to independent Brownian motions.
Hence, by Cantor's diagonal argument, we can extract a subsequence of $(N_k, k \geq 1)$ (that we still denote $(N_k, k \geq 1)$ in the following) such that for all $\Dc'\Xc' \in \Sc'$, the joint distribution
\begin{equation}
\label{eq:proof_lem_uncountable}
\left( \mu^{\Dc,a}_{\Xc;N_k}, \left( N_k^{-1} X^D_{x;N_k}(N_k^2 t), t \leq N_k^{-2} \tau^{N_k}_{D,x} \right)_{(D,x) \in \Dc\Xc} \right), \Dc\Xc \subset \Dc'\Xc',
\end{equation}
converges as $k \to \infty$.

We will conclude the proof with the following two steps.
\begin{enumerate}
\item[\textit{(i)}]
We will first fix $D_i \in \mathfrak{D}, i =1 \dots r$ and show that the fact that for all $x_i \in D_i \cap \Q^2, i=1 \dots r$, \eqref{eq:proof_lem_uncountable} converges with $\Dc'\Xc' = \{(D_i,x_i)\}$ implies the same statement for all $x_i \in D_i, i=1 \dots r$.
\item[\textit{(ii)}]
We will then fix nice domains $D_i$ and initial points $x_i \in D_i, i=1 \dots r$, and we will show that the fact that for all $D_i' \in \mathfrak{D}$ containing $x_i, i=1 \dots r$, \eqref{eq:proof_lem_uncountable} converges with $\Dc'\Xc' = \{(D_i',x_i)\}$ implies that for all pairwise distinct nice points $z_i \in \partial D_i$ and $\Dc'\Xc'\Zc' = \{ (D_i,x_i,z_i) \}$,
$\left( \mu^{\Dc,a}_{\Xc,\Zc;N_k}, \Dc\Xc\Zc \subset \Dc'\Xc'\Zc' \right)$
converges as $k \to \infty$. 
\end{enumerate}
We will only prove \textit{(ii)} since \textit{(i)} is very similar. See the end of the proof for a few comments about the step \textit{(i)} above. To ease notations, we will moreover only prove \textit{(ii)} for $r=1$. The general case $r \geq 1$ follows along the same lines by considering multivariate Laplace transforms.

Let $D$ be a nice domain, $x_0 \in D$ and $z \in \partial D$ be a nice point. Let $(X_t)_{t \geq 0}$ be the associated random walk. We assume that we already know that for all $D' \in \mathfrak{D}$ containing $x_0$, the joint distribution of
\[
\mu^{D',a}_{x_0;N_k}, \left( N_k^{-1} X_{N_k^2 t}, t \leq N_k^{-2} \tau_{D'}^{N_k} \right)
\]
converges as $k \to \infty$ and we want to show the convergence of $\mu^{D,a}_{x_0,z;N_k}, k \geq 1$. Let $f \in C_c(D, [0,\infty))$. Our objective is to show that $\scalar{\mu^{D,a}_{x_0,z;N_k},f}, k \geq 1,$ converges in law. Let $p \geq 1$ and consider $D^p \in \mathfrak{D}$ such that
\[
\{ x \in D: \mathrm{dist}(x,\partial D) \geq 2^{-p} \} \subset D^p \subset \{ x \in D: \mathrm{dist}(x,\partial D) \geq 2^{-p-1} \}.
\]
In the following, we will consider the measure $\mu_{x_0,z;N}^{D^p,a}$ which is defined as $\mu_{x_0;N}^{D^p,a}$ but under the conditional probability $\prob_{Nx_0,Nz}^{D_N^p}$ instead of $\prob_{Nx_0}^{D^p_N}$.
$(B_t, t \leq \tau_{\partial D^p})$ under $\prob^D_{x_0}$ and $(B_t, t \leq \tau_{\partial D^p})$ under $\prob^D_{x_0,z}$ are mutually absolutely continuous: if $\Fc_{\tau_{\partial D^p}}$ denotes the $\sigma$-algebra generated by $(B_t, t \leq \tau_{\partial D^p})$, we have (see \cite{AidekonHuShi2018} (2.7) for instance)
\[
\frac{d \prob^D_{x_0,z}}{d \prob^D_{x_0}} \Big\vert_{\Fc_{\tau_{\partial D^p}}} = \frac{H_D(B_{\tau_{\partial D^p}},z)}{H_D(x_0,z)} =: \Hc.
\]
Similarly (direct consequence of Markov property),
\begin{equation}
\label{eq:121b}
\frac{d \prob^{D_N}_{Nx_0,Nz}}{d \prob^{D_N}_{Nx_0}} \Big\vert_{\Fc_{\tau_{\partial D^p_N}}} = \frac{H_N(X_{\tau_{D^p}^N},Nz)}{H_N(Nx_0,Nz)} =: \Hc_N.
\end{equation}
Hence the convergence of $ \left(\scalar{\mu^{D^p,a}_{x_0;N_k},f}, X_{\tau_{D^p}^{N_k}}/N_k \right), k \geq 1$, implies the convergence of
\newline
$\scalar{\mu^{D^p,a}_{x_0,z;N_k},f}$, $k \geq 1$: by Lemma \ref{lem:Poisson}, for all $\alpha>0$ and $k$ large enough,
\begin{align*}
\Expect{\exp \left( - \scalar{\mu^{D^p,a}_{x_0,z;N_k},f} \right)}
& =
\Expect{\Hc_{N_k}\exp \left( - \scalar{\mu^{D^p,a}_{x_0;N_k},f} \right)} \\
& \leq
\Expect{ \left( \frac{H_D \left( X_{\tau_{D^p}^{N_k}}/N_k , z \right)}{H_D(x_0,z)} + \alpha \right) \exp \left( - \scalar{\mu^{D^p,a}_{x_0;N_k},f} \right)} \\
& \xrightarrow[k \to \infty]{} \Expect{(\Hc+\alpha) \exp \left( - \scalar{\mu^{D^p,a}_{x_0},f} \right)}
\end{align*}
and
\[
\limsup_{k \to \infty} \Expect{\exp \left( - \scalar{\mu^{D^p,a}_{x_0,z;N_k},f} \right)} \leq \Expect{\Hc \exp \left( - \scalar{\mu^{D^p,a}_{x_0},f} \right)}.
\]
We obtain similarly that the liminf is bounded from below by the above right hand side term implying that $\Expect{\exp \left( - \scalar{\mu^{D^p,a}_{x_0,z;N_k},f} \right)}$ converges as $k \to \infty$.
Since $D^p, p \geq 1,$ is an increasing sequence of domains, for all $k \geq 1$, $\Expect{\exp \left( - \scalar{\mu^{D^p,a}_{x_0,z;N_k},f} \right)}$ is non-increasing with $p$. Hence
\[
\lim_{k \to \infty} \Expect{\exp \left( - \scalar{\mu^{D^p,a}_{x_0,z;N_k},f} \right)}
\]
converges when $p \to \infty$.
By Lemma \ref{lem:1-point}, we also notice that for all $N \geq 1$ and $p \geq 1$,
\[
0 \leq \Expect{ \scalar{\mu_{x_0,z;N}^{D,a},f}} - 
\Expect{ \scalar{\mu_{x_0,z;N}^{D^p,a},f}} \leq o_{p \to \infty}(1).
\]
By lemma \ref{lem:elementary}, it implies that $ \scalar{\mu_{x_0,z;N_k}^{D,a},f}, k \geq 1,$ converges in distribution. This concludes the proof of the step \textit{(ii)}.

We finish this proof with a comment about the step \textit{(i)}. The proof is very similar. One would need to first stop the walks at the first hitting times of small discs centred at the starting points $x_i$. One would need to argue that the main contribution comes from the rest of the trajectories which converge by an $h$-transform-type of argument as above. We leave the details to the reader.
\end{proof}

As mentioned in Section \ref{subsec:organisation}, to prove that the subsequential limits satisfy Properties \ref{charac1} and \ref{charac3}, we need the following result which is proven in Section \ref{sec:uniform integrability}:

\begin{proposition}[Uniform integrability]\label{prop:uniform_integrability}
For all $\Dc\Xc\Zc \in \Sc$ and $K \Subset \bigcap_{D \in \Dc} D$, 
\[
\left( \mu_{\Xc;N}^{\Dc,a}(\C), N \geq 1 \right)
\mathrm{~and~}
\left( \mu^{\Dc,a}_{\Xc,\Zc;N}(K), N \geq 1 \right)
\]
are uniformly integrable. Moreover, any subsequential limit $\mu^{\Dc,a}_{\Xc,\Zc}$ of $\left( \mu^{\Dc,a}_{\Xc,\Zc;N}, N \geq 1 \right)$ satisfies: almost surely for all Borel set $A$ of Hausdorff dimension less than $2-a$, $\mu^{\Dc,a}_{\Xc,\Zc}(A) = 0$.
\end{proposition}

Before jumping into the proof of Theorem \ref{th:convergence}, we state the following result which is a quick consequence of \eqref{eq:lem_Poisson_convergence}.

\begin{lemma}\label{lem:elementary2}
Let $x_0 \in U$ and let $\phi_N : \C \to [0,1]$ be a sequence of functions converging pointwise towards $\phi$.
Let $\{z_i, i =1 \dots p \} \subset \partial U$ be the points where the boundary $\partial U$ is not analytic. Assume that for all $\alpha >0$ and for any compact subset $K$ of $\C \backslash \{z_i,i=1 \dots p\}$, there exists $C_{\alpha,K} >0$ such that for all $N$ large enough and for all $z,z' \in K$,
\[
\abs{\phi_N(z) - \phi_N(z')} \leq C_{\alpha,K} \abs{z-z'} + \alpha.
\]
Then,
\[
\sum_{z \in \partial U_N} H^{U_N}( \floor{Nx_0}, z) \phi_N(z/N) \xrightarrow[N \to \infty]{} \int_{\partial U} H^U(x_0,z) \phi(z) dz.
\]
$dz$ denotes here the one-dimensional Hausdorff measure on $\partial U$.
\end{lemma}

\begin{proof}
In this proof, when we say that a set $K \subset \C$ is smooth, we mean that each connected component of the boundary of $K$ is analytic.
Let $\alpha,\eps>0$. Since $0 \leq \phi \leq 1$, there exists a smooth compact subset $K$ of $\C \backslash \{z_i,i=1 \dots p\}$ such that
\[
\int_{\partial U \backslash K} H^U(x,z) \phi(z) dz \leq \alpha.
\]
Using the weak convergence \eqref{eq:lem_Poisson_convergence}, this upper bound in particular implies
\[
\limsup_{N \to \infty} \sum_{z \in \partial U_N} \indic{z/N \notin K} H^{U_N}( \floor{Nx}, z) \phi_N(z/N) \leq \alpha.
\]
We now decompose $K = \cup_{i=1}^I K_i$ into smooth compact sets of diameter at most $\eps$ and such that for all $i \neq j$, $K_i \cap K_j \cap \partial U$ is composed of at most one point. For all $i = 1 \dots I$, let $y_i$ be any point of $K_i$. By the weak convergence \eqref{eq:lem_Poisson_convergence}, we now have
\begin{align*}
& \limsup_{N \to \infty} \sum_{z \in \partial U_N} \indic{z/N \in K} H^{U_N}( \floor{Nx}, z) \phi_N(z/N) \\
& \leq \alpha + C_{\alpha,K} ~\eps + \limsup_{N \to \infty} \sum_{i=1}^I \phi_N(y_i) \sum_{z \in \partial U_N} \indic{z/N \in K_i} H^{U_N} (\floor{Nx},z) \\
& \leq \alpha + C_{\alpha,K} ~\eps + \sum_{i=1}^I \phi(y_i) \int_{\partial U \cap K_i} H^U(x,z) dz\\
& \leq 2 \alpha + 2 C_{\alpha,K} ~\eps + \int_{\partial U \cap K} H^U(x,z) \phi(z) dz.
\end{align*}
We have obtained
\[
\limsup_{N \to \infty} \sum_{z \in \partial U_N} H^{U_N}( \floor{Nx}, z) \phi_N(z/N)
\leq 3 \alpha + 2 C_{\alpha,K} ~\eps + \int_{\partial U} H^U(x,z) \phi(z) dz.
\]
We obtain the desired upper bound by letting $\eps \to 0$ and then $\alpha \to 0$. The lower bound is similar.
\end{proof}

We are now ready to prove Theorem \ref{th:convergence}.

\begin{proof}[Proof of Theorem \ref{th:convergence}]
Let $x_0 \in U$. We start by assuming the convergence of $(\mu_{x_0,z;N}^{U,a}, N \geq 1)$ for all nice points $z \in \partial U$ and we are going to explain how we deduce the convergence of $(\mu_{x_0;N}^{U,a}, N \geq 1)$. Let $f \in C(D,[0,\infty))$. It is enough to prove that
\[
\Expect{ \exp \left( - \scalar{ \mu_{x_0;N}^{U,a},f} \right) }
\]
converges.
By Lemma \ref{lem:1-point} \eqref{eq:lem_bound},
\[
\lim_{r \to 0} \sup_N \Expect{ \mu_{x_0;N}^{U,a} \left( \{x \in U: d(x,\partial U) \leq r\} \right) } = 0.
\]
We can thus assume that $f$ has a compact support included in $U$ (see Lemma \ref{lem:elementary}). We have
\[
\Expect{ \exp \left( - \scalar{ \mu_{x_0;N}^{U,a},f} \right) }
= \sum_{z \in \partial U_N} H^{U_N}(x_0,\floor{Nz}) \Expect{ \exp \left( - \scalar{ \mu_{x_0,z/N;N}^{U,a},f} \right) }.
\]
To obtain the convergence of the above sum, we are going to show that we can cast our situation into Lemma \ref{lem:elementary2}. Let $\alpha, r>0$ and define
\[
U^r := \{ x \in U: \mathrm{dist}(x,\partial U) > r \}.
\]
By Lemma \ref{lem:1-point}, if $r$ is small enough (possibly depending on $U, x_0$ and $f$), we have for all $z \in \partial D$,
\begin{align*}
& \abs{ \Expect{ \exp \left( - \scalar{ \mu_{x_0,z;N}^{U,a},f} \right) } - \Expect{ \left. \exp \left( - \scalar{ \mu_{x_0;N}^{U^r,a},f} \right) \right\vert X_{\tau_{\partial U_N}} = \floor{Nz} } } \\
& \leq \Expect{ \scalar{ \mu_{x_0,z;N}^{U,a},f} } - \Expect{ \left. \scalar{ \mu_{x_0;N}^{U^r,a},f} \right\vert X_{\tau_{\partial U_N}} = \floor{Nz}  } \leq \alpha/3.
\end{align*}
We now notice by Lemma \ref{lem:Poisson} \eqref{eq:lem_Poisson} that for all $N$ large enough and $z,z' \in \partial D$,
\begin{align*}
& \abs{ \Expect{ \left. \exp \left( - \scalar{ \mu_{x_0;N}^{U^r,a},f} \right) \right\vert X_{\tau_{\partial U_N}} = \floor{Nz} } - \Expect{ \left. \exp \left( - \scalar{ \mu_{x_0;N}^{U^r,a},f} \right) \right\vert X_{\tau_{\partial U_N}} = \floor{Nz'} } } \\
& = \abs{ \Expect{ \exp \left( - \scalar{ \mu_{x_0;N}^{U^r,a},f} \right) \left( \frac{H^{U_N}(X_{\tau_{\partial U_N^r}},\floor{Nz})}{H^{U_N}(x_0,\floor{Nz})} - \frac{H^{U_N}(X_{\tau_{\partial U_N^r}},\floor{Nz'})}{H^{U_N}(x_0,\floor{Nz'})} \right) } } \\
& \leq \alpha/3 + \sup_{x,y \in U^r} \abs{ \frac{H^U(x,z)}{H^U(y,z)} - \frac{H^U(x,z')}{H^U(y,z')} }.
\end{align*}
Using \eqref{eq:Poisson_conformal}, we see that for all compact subset $K$ of an analytic portion of $\partial U$, the above supremum is at most $C_{\alpha,K} \abs{z-z'}$ for all $z,z' \in K$. We have proven that for all $N$ large enough, all such compact subset $K$ and $z,z' \in K$,
\[
\abs{ \Expect{ \exp \left( - \scalar{ \mu_{x_0,z;N}^{U,a},f} \right) } - \Expect{ \exp \left( - \scalar{ \mu_{x_0,z';N}^{U,a},f} \right) } } \leq C_{\alpha,K} \abs{z-z'} + \alpha.
\]
 We can thus conclude with Lemma \ref{lem:elementary2} that
\[
\Expect{ \exp \left( - \scalar{ \mu_{x_0;N}^{U,a},f} \right) }
\xrightarrow[N \to \infty]{} \int_{\partial D} H^U(x_0,z) \lim_{N \to \infty} \Expect{ \exp \left( - \scalar{ \mu_{x_0,z;N}^{U,a},f} \right) } dz.
\]
This finishes the transfer of the convergence of conditioned measures to unconditioned measures.

We now turn to the proof of the convergence of $(\mu_{x_*,z_*;N}^{U,a}, N \geq 1)$ where $x_* \in U$ and $z_* \in \partial U$ is a nice point. Let $(N_k, k \geq 1)$ be an increasing sequence of integers such that $(\mu_{x_*,z_*;N_k}^{U,a}, k \geq 1)$ converges. By Lemma \ref{lem:uncountable_extraction}, by extracting a further subsequence if necessary, we can assume that for all
$\Dc'\Xc'\Zc' \in \Sc$,
\[
(\mu_{\Xc,\Zc;N_k}^{\Dc,a}, \Dc\Xc\Zc \subset \Dc'\Xc'\Zc')
\]
converges as $k \to \infty$ towards some
\[
(\mu_{\Xc,\Zc}^{\Dc,a}, \Dc\Xc\Zc \subset \Dc'\Xc'\Zc').
\]
By Theorem \ref{th:charac}, to show that $\mu_{x_*,z_*}^{U,a} \overset{\mathrm{(d)}}{=} e^{c_0a/g} \Mc_{x_*,z_*}^{U,a}$, it is enough to prove that 
$(\mu^{\Dc,a}_{\Xc,\Zc}$, $\Dc\Xc\Zc \in \Sc)$
satisfies Properties \ref{charac1}-\ref{charac3}.

Property \ref{charac1} is a direct consequence of what we have already done.
For instance, for $\Dc\Xc\Zc = \{(D,x_0,z)\} \in \Sc$, the arguments are as follows. In order to identify the two finite Borel measures
\[
\Expect{ \mu^{D,a}_{x_0,z}(dx)}
\quad \text{and} \quad
e^{c_0a/g} \psi^D_{x_0,z}(x) dx,
\]
we only need to check that for any continuous bounded nonnegative function $f : \C \to \R$, the integrals of $f$ against these two measures agree. For $r>0$, let $f_r$ be a continuous function with support compactly included in $D$ which agrees with $f$ on 
$\{x \in D: d(x,\partial D) \geq r \}$ and such that $0 \leq f_r \leq f$.
By Proposition \ref{prop:tightness}, for all $r>0$,
\[
\lim_{k \to \infty} \Expect{\scalar{ \mu^{D,a}_{x_0,z;N_k}, f_r } } = e^{c_0a/g} \int_D f_r(x) \psi^{D,a}_{x_0,z}(x) dx.
\]
Since Proposition \ref{prop:uniform_integrability} shows that $(\scalar{\mu^{D,a}_{x_0,z;N_k}, f_r}, k \geq 1)$ is uniformly integrable, we can interchange the limit and the expectation which gives
\[
\Expect{\scalar{\mu^{D,a}_{x_0,z}, f_r} } = e^{c_0a/g} \int_D f_r(x) \psi^D_{x_0,z}(x) dx.
\]
We then obtain Property \ref{charac1} by letting $r \to 0$ and using monotone convergence theorem.

The proof of Property \ref{charac2} is very similar to the Brownian case. For instance, in the case $\Dc\Xc\Zc = \{(D,x_0,z)\} \in \Sc$ and $D'$ nice subset of $D$ containing $x_0$, we can very similarly show that for all continuous function $f : \C \to [0,\infty)$ with compact support included in $D \backslash \partial D'$, and all $y \in \partial D'$, $\scalar{\mu_{x_0,z;N}^{D,a},f}$ under $\prob^{D_N}_{\floor{Nx_0},\floor{Nz}} \left( \cdot \left\vert X_{\tau_{\partial D'_N}} = \floor{Ny} \right. \right)$ has the same law as
\[
\scalar{\mu_{x_0,y;N}^{D',a},f} + \scalar{\mu_{y,z;N}^{D,a},f} + \scalar{ \mu_{(D',x_0,y),(D,y,z);N}^{a}, f}
\]
plus smaller order terms which converge to zero in $L^1$. This shows the conditional version of Property \ref{charac2}. To obtain Property \ref{charac2} without having to condition on the hitting point of $\partial D'$, we have to integrate over $y \in \partial D'$. For this, we use the same argument as what we did at the very beginning of the proof to transfer results from the conditioned to the unconditioned measures.

Finally, Property \ref{charac2bis} follows from the fact that we consider independent random walks and Property \ref{charac3} is a direct consequence of the carrying dimension estimate of Proposition \ref{prop:uniform_integrability}. This concludes the proof.
\end{proof}

\section{Uniform integrability: proof of Proposition \ref{prop:uniform_integrability}}
\label{sec:uniform integrability}

To ease notations, we will prove Proposition \ref{prop:uniform_integrability} for $\Dc\Xc\Zc = \{(D,x_0,z)\}$.
Our approach is very close to the one of \cite{jegoGMC}. We have simplified some minor aspects since we only need to show the uniform integrability of the sequence but not its convergence in $L^1$. For instance, our definition of ``good events'' limits the number of certain excursions rather than limiting certain local times. 

If $x \in \Z^2$ and $R \geq 1$, we will denote by $C_R(x)$ the contour $\Z^2 \cap \partial (x + [-R,R]^2)$, by $A_N(x \to R)$ the number of excursions from $x$ to $C_R(x)$ before $\tau_{\partial D_N}$ and
\begin{equation}
\label{eq:app_def_qR}
q_R := \left.\log \left(\frac{N}{R} \right) \right/ \log N.
\end{equation}
For $b \in (a,2)$ and $\eps>0$, we introduce 
\[
D^\eps := \{ x \in D: d_\infty(x,\partial D) > 2 \eps \mathrm{~and~} \abs{x-x_0} \geq 2\eps \},
\]
the good event at $x$
\[
G^{b,\eps}_N(x) := \left\{ \forall R \in (2^p)_{p \geq 1} \cap [N^{1/2-a/4},\eps N], A_N(x \to R) \leq \frac{b}{2} \frac{1+q_R}{1-q_R} \log \frac{N}{R} \right\}
\]
and the modified version of $\mu_{x_0;N}^{D,a}(\C)$,
\[
\bar{\mu}_{x_0;N}^{D,a}(\C) := \frac{\log N}{N^{2 - a}} \sum_{x \in \Z^2} \indic{x/N \in D^\eps} \mathbf{1}_{G_N^{b,\eps}(x)} \indic{ \ell_x^{\tau_{\partial D_N}} \geq g a \log^2 N}.
\]
We will see that adding these good events does not change the behaviour of the first moment and makes the second moment finite.

\begin{lemma}\label{lem:app_first_moment_unchanged}
For all $b >a$,
\[
\lim_{\eps \to 0} \sup_{N \geq 1} \Expect{ \abs{\mu_{x_0;N}^{D,a}(\C) - \bar{\mu}_{x_0;N}^{D,a}(\C)} } = 0.
\]
\end{lemma}

\begin{lemma}\label{lem:app_second_moment}
If $b>a$ is close enough to $a$,
\begin{equation}
\label{eq:lem_app_second_moment}
\sup_{N \geq 1} \Expect{ \bar{\mu}_{x_0;N}^{D,a}(\C)^2 } < \infty.
\end{equation}
Moreover, if $b$ is close enough to $a$, for all $\eta >0$,
\begin{equation}
\label{eq:lem_app_energy}
\sup_{N \geq 1} \Expect{ \int_{\C^2} \frac{1}{\abs{x-y}^{2-2b+a-\eta}} \bar{\mu}_{x_0;N}^{D,a}(dx) \bar{\mu}_{x_0;N}^{D,a}(dy) } < \infty.
\end{equation}
\end{lemma}

We now explain how these two lemmas imply Proposition \ref{prop:uniform_integrability}.

\begin{proof}[Proof of Proposition \ref{prop:uniform_integrability}]
Lemma \ref{lem:app_first_moment_unchanged} and \eqref{eq:lem_app_second_moment} imply that $(\mu_{x_0;N}^{D,a}(\C), N \geq 1)$ is uniformly integrable. Moreover, by Frostman's lemma, Lemma \ref{lem:app_first_moment_unchanged} and the energy estimate \eqref{eq:lem_app_energy} imply that any subsequential limit $\mu_{x_0}^{D,a}$ of $\mu_{x_0;N}^{D,a}, N \geq 1$, satisfies: almost surely for all Borel set $A$ with Hausdorff dimension smaller than $2-a$, $\mu_{x_0}^{D,a}(A) = 0$.

To finish the proof, we now have to explain how we transfer these results to the conditioned measures $\mu_{x_0,z;N}^{D,a}$, $N \geq 1$.
Let $K \Subset D$, $r >0$ and define $D^r := \{ x \in D, d(x, \partial D) > r \}$.
We denote by $\mu_{x_0,z;N}^{D^r,a}(K)$ the random variable
\[
\frac{\log N}{N^{2-a}} \sum_{x \in \Z^2} \indic{x/N \in K} \indic{ \ell_x^{\tau_{\partial D^r_N}} \geq ga \log^2 N}
\]
under $\prob_{x_0,z}^{D_N}$. A similar reasoning as in the proof of Lemma \ref{lem:uncountable_extraction} shows that
\[
0 \leq \Expect{ \mu_{x_0,z;N}^{D,a}(K) - \mu_{x_0,z;N}^{D^r,a}(K) } \leq p(r)
\]
for some $p(r) >0$ which may depend on $a,D,x_0,z$ and which goes to zero as $r \to 0$. Hence, to show the uniform integrability of $(\mu_{x_0,z;N}^{D,a}(K), N \geq 1)$, it is enough to show that $(\mu_{x_0,z;N}^{D^r,a}(K), N \geq 1)$ is uniformly integrable. Recalling that (see \eqref{eq:121b}, \eqref{eq:lem_Poisson} and \eqref{eq:Poisson_conformal})
\[
\frac{d \prob^{D_N}_{Nx_0,Nz}}{d \prob^{D_N}_{Nx_0}} \Big\vert_{\Fc_{\tau_{\partial D^r_N}}} = \frac{H_N(X_{\tau_{\partial D^r_N}},Nz)}{H_N(Nx_0,Nz)} \in [\alpha, 1/ \alpha]
\]
for some $\alpha = \alpha(r) \in (0,1)$, we then observe that for all $M >0$,
\begin{align*}
\Expect{ \mu_{x_0,z;N}^{D^r,a}(\C) \indic{ \mu_{x_0,z;N}^{D^r,a}(\C) \geq M} }
& \leq \frac{1}{\alpha} \Expect{ \mu_{x_0;N}^{D^r,a}(\C) \indic{ \mu_{x_0;N}^{D^r,a}(\C) \geq M} } \\
& \leq \frac{1}{\alpha} \Expect{ \mu_{x_0;N}^{D,a}(\C) \indic{ \mu_{x_0;N}^{D,a}(\C) \geq M} }.
\end{align*}
The uniform integrability of $(\mu_{x_0;N}^{D,a}(\C), N \geq 1)$ thus implies the uniform integrability of
\newline
$(\mu_{x_0,z;N}^{D^r,a}(\C)$, $N \geq 1)$. 

To obtain the carrying dimension estimate we proceed in a similar manner. If we denote $\bar{\mu}_{x_0,z;N}^{D^r,a}(dx)$ the modified version of $\mu_{x_0,z;N}^{D^r,a}(dx)$ for which we have added the good events $G^{b,\eps}_N(x)$ for the domain $D^r$, we have as before
\begin{align*}
& \sup_{N \geq 1} \Expect{ \int_{\C^2} \frac{1}{\abs{x-y}^{2-2b+a-\eta}} \bar{\mu}^{D^r,a}_{x_0,z;N}(dx) \bar{\mu}^{D^r,a}_{x_0,z;N}(dy) } \\
& \leq \frac{1}{\alpha} \sup_{N \geq 1} \Expect{ \int_{\C^2} \frac{1}{\abs{x-y}^{2-2b+a-\eta}} \bar{\mu}^{D^r,a}_{x_0;N}(dx) \bar{\mu}^{D^r,a}_{x_0;N}(dy) } < \infty
\end{align*}
and 
\[
\limsup_{\eps \to 0} \sup_{N \geq 1} \Expect{ \mu_{x_0,z;N}^{D^r,a}(\C) - \bar{\mu}^{D^r,a}_{x_0,z;N}(\C) } \leq \frac{1}{\alpha} \limsup_{\eps \to 0} \sup_{N \geq 1} \Expect{ \mu_{x_0;N}^{D^r,a}(\C) - \bar{\mu}^{D^r,a}_{x_0;N}(\C) } 0.
\]
For the same reasons as before, this shows that any subsequential limit $\mu_{x_0,z}^{D^r,a}$ of $\mu_{x_0,z;N}^{D^r,a}$, $N \geq 1$, satisfies: almost surely for all Borel set $A$ of Hausdorff dimension smaller than $2-a$, $\mu_{x_0,z}^{D^r,a}(A)=0$. Since this is true for all $r >0$, it completes the proof the carrying dimension estimate of Proposition \ref{prop:uniform_integrability}. This concludes the proof.
\end{proof}

The rest of this section is dedicated to the proofs of Lemmas \ref{lem:app_first_moment_unchanged} and \ref{lem:app_second_moment}. We now lay the groundwork.
If $A \subset \Z^2$, we will write
\[
\tau_A := \inf \{ t >0: X_t \in A \}
\]
and for $x \in \Z^2$, $\tau_x := \tau_{ \{x\} }$.
Let $N \geq 1$. For $x,y \in D_N$, we will denote
\begin{equation}
\label{eq:app_def_pxy}
p_{xy} := \PROB{x}{\tau_y < \tau_{\partial D_N}} = G^{D_N}(x,y) / G^{D_N}(y,y).
\end{equation}
If $x$ and $y$ are in the bulk of $D_N$, Lemma \ref{lem:Green} implies that
\begin{equation}
\label{eq:app_approx_pxy}
p_{xy} = q_{\abs{x-y}} \left(1 + O \left( \frac{1}{\log N} \right) \right).
\end{equation}
We start off with two easy lemmas. The first one is the analogue of \cite[Lemma 2.3]{jegoGMC} whereas the second one is well-known and a proof can be found for instance in \cite[Lemma 4.2.1]{jego2020}.

\begin{lemma}\label{lem:app_hitting_proba}
For all pairwise distinct points $x,y,z$ of $D_N$,
\[
\PROB{z}{\tau_x < \tau_y \wedge \tau_{\partial D_N} } = \frac{p_{zx}-p_{zy}p_{yx}}{1-p_{xy}p_{yx}}.
\]
\end{lemma}

\begin{proof}
By Markov property, we have
\begin{align*}
\PROB{z}{\tau_y < \tau_{\partial D_N} } & = \PROB{z}{\tau_y < \tau_x \wedge \tau_{\partial D_N} }
+ \PROB{z}{\tau_x < \tau_y < \tau_{\partial D_N} } \\
& = \PROB{z}{\tau_y < \tau_x \wedge \tau_{\partial D_N} } + \PROB{z}{\tau_x < \tau_y \wedge \tau_{\partial D_N}} \PROB{x}{\tau_y < \tau_{\partial D_N}}.
\end{align*}
By exchanging the roles of $x$ and $y$ we find that
\[
\PROB{z}{\tau_x < \tau_{\partial D_N} } = \PROB{z}{\tau_x < \tau_y \wedge \tau_{\partial D_N} } + \PROB{z}{\tau_y < \tau_x \wedge \tau_{\partial D_N}} \PROB{y}{\tau_x < \tau_{\partial D_N}}.
\]
Combining these two equalities yields the stated claim.
\end{proof}

\begin{lemma}\label{lem:app_indep_local_time}
For all subset $A \subset \Z^2$ and $x \in \Z^2$, starting from $x$, $\ell_x^{\tau_A}$ is an exponential variable independent of $X_{\tau_A}$.
\end{lemma}

We now fix $x,y \in D_N$, $R \in (2^p)_{p \geq 1} \cap [N^{1/2-a/4},\eps N]$ such that $x/N,y/N \in D^\eps$ and such that $y \notin x + [-R,R]^2$ and we describe the joint law of $(\ell_x^{\tau_{\partial D_N}},\ell_y^{\tau_{\partial D_N}}, A_N(x \to R), A_N(y \to R))$.
For $i\geq 1$, we denote by $\ell_x^i$ (resp. $\ell_y^i$) the local time at $x$ (resp. $y$) accumulated during the $i$-th excursion from $x$ to $C_R(x)$ (resp. from $y$ to $C_R(y)$). We have
\begin{equation}
\label{eq:app_joint_law}
\ell_x^{\tau_{\partial D_N}} = \sum_{i=1}^{A_N(x\to R)} \ell_x^i
\mathrm{~and~}
\ell_y^{\tau_{\partial D_N}} = \sum_{i=1}^{A_N(y\to R)} \ell_y^i.
\end{equation}
By Markov property and by Lemma \ref{lem:app_indep_local_time}, conditioned on $A_N(x \to R)$ and $A_N(y \to R)$,
the variables $\ell_x^i, i =1 \dots A_N(x \to R)$ and $\ell_y^i, i =1 \dots A_N(y \to R)$ are i.i.d. exponential random variables with mean equal to
\begin{equation}
\label{eq:app_9}
\EXPECT{x}{\ell_x^{\tau_{C_R(x)}}} = \left( 1 + O \left( \frac{1}{\log N} \right) \right) g \log R = \left( 1 + O \left( \frac{1}{\log N} \right) \right) g(1-q_R) \log N.
\end{equation}
Moreover, by \eqref{eq:app_approx_pxy}, for all $k \geq 1$,
\begin{align}
\PROB{Nx_0}{A_N(x \to R) \geq k} & = \PROB{Nx_0}{\ell_x^{\tau_{\partial D_N}} >0} \PROB{x}{ \ell_x^{\tau_{\partial D_N}} - \ell_x^{\tau_{C_R(x)}} >0}^{k-1} \nonumber \\
& = \left( 1 + O \left( \frac{1}{\log N} \right) \right)^{k-1} \PROB{Nx_0}{\ell_x^{\tau_{\partial D_N}} >0} q_R^{k-1}.
\label{eq:app_6}
\end{align}
Similarly, we notice that if $c \abs{x-y} \leq R \leq \abs{x-y} / 10$, then Lemma \ref{lem:app_hitting_proba} and \eqref{eq:app_approx_pxy} show that
\begin{align}
\label{eq:app_11}
& \PROB{Nx_0}{A_N(x\to R) + A_N(y\to R) = k} \\
& \leq \PROB{Nx_0}{\tau_x \wedge \tau_y < \tau_{\partial D_N}}
\left( \frac{2q_R}{1+q_R} \right)^{k-1} \left( 1 + O \left( \frac{1}{\log N} \right) \right)^{k-1}. \nonumber
\end{align}

%

Finally, we state for ease of reference the following two elementary inequalities:
\begin{align}
\label{eq:app_elem1}
\mathrm{if~} \mu \leq 1&, \sum_{i=n}^\infty \frac{(\mu n)^i}{i!} \leq (\mu e)^n,\\
\label{eq:app_elem2}
\mathrm{if~} \mu \geq 1&, \sum_{i=0}^{n-1} \frac{(\mu n)^i}{i!} \leq e(\mu e)^{n-1}.
\end{align}
We will moreover denote $\Gamma(k,1)$ a Gamma random variable with shape parameter $k$ and scale parameter $1$. This variable has the same law as the sum of $k$ independent exponential variables with parameter 1. Recall that for all $k,k' \geq 1$ and $t >0$,
\begin{equation}
\label{eq:app_gamma}
\Prob{ \Gamma(k,1) > t } = e^{-t} \sum_{i=0}^{k-1} \frac{t^i}{i!}
\end{equation}
and
\begin{align}
\nonumber
\Prob{\Gamma(k,1) \geq t} \Prob{\Gamma(k',1) \geq t}
& = e^{-2t} \sum_{n=0}^{k+k'-2} t^n \sum_{\substack{0 \leq i \leq k-1\\0 \leq j \leq k'-1\\i+j=n}} \frac{1}{i!j!} \\
& \leq e^{-2t} \sum_{n=0}^{k+k'-2} t^n \sum_{\substack{i,j \geq 0\\i+j=n}} \frac{1}{i!j!}
= e^{-2t} \sum_{n=0}^{k+k'-2} \frac{(2t)^n}{n!}
\label{eq:app_10}
\end{align}

We are now ready to prove Lemmas \ref{lem:app_first_moment_unchanged} and \ref{lem:app_second_moment}.

\begin{proof}[Proof of Lemma \ref{lem:app_first_moment_unchanged}]
Firstly, by Lemma \ref{lem:1-point},
\[
\lim_{\eps \to 0} \sup_{N \geq 1} \expect_{Nx_0}^{D_N} \left[ \frac{\log N}{N^{2 - a}} \sum_{x \in \Z^2} \indic{x/N \notin D^\eps} \indic{ \ell_x^{\tau_{\partial D_N}} \geq g a \log^2 N} \right] = 0.
\]
So we only need to show that
\begin{equation}
\label{eq:app_goal1}
\lim_{\eps \to 0} \sup_{N \geq 1} \expect_{Nx_0}^{D_N} \left[ \frac{\log N}{N^{2 - a}} \sum_{x \in \Z^2} \indic{x/N \in D^\eps} \mathbf{1}_{G_N^{b,\eps}(x)^c} \indic{ \ell_x^{\tau_{\partial D_N}} \geq g a \log^2 N} \right] = 0.
\end{equation}
Let $x \in \Z^2$ s.t. $x/N \in D^\eps$. By a union bound,
\begin{align}
\label{eq:app_2}
& \prob_{Nx_0}^{D_N} \left( \ell_x^{\tau_{\partial D_N}} \geq g a \log^2 N, G_N^{b,\eps}(x)^c \right) \\
& \leq \sum_{\substack{R \in (2^p)_{p \geq 1}\\N^{1/2-a/4} \leq R \leq \eps N}} \prob_{Nx_0}^{D_N} \left( \ell_x^{\tau_{\partial D_N}} \geq g a \log^2 N, A_N(x \to R) > \frac{b}{2} \frac{1+q_R}{1-q_R} \log \frac{N}{R} \right). \nonumber
\end{align}
Let $R \in (2^{-p})_{p \geq 1} \cap [N^{1/2-a/4}, \eps N]$. 
In the discussion following Lemma \ref{lem:app_hitting_proba} we described the joint law of $(\ell_x^{\tau_{\partial D_N}}, A_N(x \to R))$. Using the notations therein and by \eqref{eq:app_gamma}, we have
\begin{align*}
& \prob_{Nx_0}^{D_N} \left( \ell_x^{\tau_{\partial D_N}} \geq g a \log^2 N, A_N(x \to R) > \frac{b}{2} \frac{1+q_R}{1-q_R} \log \frac{N}{R} \right) \\
& = O(1) \prob_{Nx_0}^{D_N} \left( \ell_x^{\tau_{\partial D_N}} >0 \right) (1- q_R) \\
& \times \sum_{k > \frac{b}{2} \frac{1+q_R}{1-q_R} \log \frac{N}{R}} \left( 1 + O \left( \frac{1}{\log N} \right) \right)^{k-1} q_R^{k-1} \Prob{ \Gamma(k,1) \geq \frac{a\log N}{1-q_R} \left( 1 + O \left( \frac{1}{\log N} \right) \right) } \\
& = O(1) \frac{G^{D_N}(Nx_0,x)}{G^{D_N}(x,x)} (1- q_R) e^{-a\log N/(1-q_R) } \\
& \times \sum_{k > \frac{b}{2} \frac{1+q_R}{1-q_R} \log \frac{N}{R}} \left( 1 + O \left( \frac{1}{\log N} \right) \right)^{k-1} q_R^{k-1} \sum_{i=0}^{k-1} \frac{1}{i!} \left( \frac{a\log N}{1-q_R} \right)^i \\
& = O(1) \frac{G^{D_N}(Nx_0,x)}{G^{D_N}(x,x)} e^{-a\log N/(1-q_R) } \Bigg( q_R^{\frac{b}{2} \frac{1+q_R}{1-q_R} \log \frac{N}{R}} \sum_{i = 0}^{\frac{b}{2} \frac{1+q_R}{1-q_R} \log \frac{N}{R}-1} \frac{1}{i!} \left( \frac{a\log N}{1-q_R} \right)^i \\
& ~~~~~~~~~~~~~~~~~~~~~~~~~~~~~~~~~~~~~~~~~~~~~~+ \sum_{i \geq \frac{b}{2} \frac{1+q_R}{1-q_R} \log \frac{N}{R}} \frac{1}{i!} \left( q_R \frac{a\log N}{1-q_R} \right)^i \left( 1 + O \left( \frac{1}{\log N} \right) \right)^i
\Bigg)
\end{align*}
We are going to bound each individual term of the above expression. Let
\[
q_{ab} := \sup \left\{ q \in (0,1): \frac{a}{q} \geq \frac{b(1+q)}{2} \right\} < 1.
\]
There exists $\eta = \eta(a,b)$ such that for all $q \in [q_{ab}, 1], \log q \leq q-1 - \eta (q-1)^2$. We deduce that if $q_R \in [q_{ab},1]$,
\begin{align*}
N^a q_R^{\frac{b}{2} \frac{1+q_R}{1-q_R} \log \frac{N}{R}} & = \exp \left[ \left( \frac{a}{q_R} + \frac{b}{2} \frac{1+q_R}{1-q_R} \log q_R \right) \log \frac{N}{R} \right] \\
& \leq \exp \left[ \left( \frac{a}{q_R} -\frac{b(1+q_R)}{2} - \frac{\eta b(1-q_R^2)}{2} \right) \log \frac{N}{R} \right] 
\leq \exp \left[ - \eta' \log \frac{N}{R} \right]
\end{align*}
for some $\eta' = \eta'(a,b)>0$. Hence, if $q_R \in [q_{ab},1]$, we have
\[
e^{-a\log N/(1-p_R) } q_R^{\frac{b}{2} \frac{1+q_R}{1-q_R} \log \frac{N}{R}} \sum_{i = 0}^{\frac{b}{2} \frac{1+q_R}{1-q_R} \log \frac{N}{R}-1} \frac{1}{i!} \left( \frac{a\log N}{1-q_R} \right)^i
\leq q_R^{\frac{b}{2} \frac{1+q_R}{1-q_R} \log \frac{N}{R}} \leq N^{-a} \left( \frac{N}{R} \right)^{-\eta'}.
\]
If $q_R < q_{ab}$, we use \eqref{eq:app_elem2} and we get
\begin{align*}
& e^{-a\log N/(1-q_R) } q_R^{\frac{b}{2} \frac{1+q_R}{1-q_R} \log \frac{N}{R}} \sum_{i = 0}^{\frac{b}{2} \frac{1+q_R}{1-q_R} \log \frac{N}{R}-1} \frac{1}{i!} \left( \frac{a\log N}{1-q_R} \right)^i \\
& \leq O(1) e^{-a\log N/(1-q_R) } \left( \frac{2ae}{b(1+q_R)} \right)^{\frac{b}{2} \frac{1+q_R}{1-q_R} \log \frac{N}{R}} \\
& = O(1) N^{-a} \exp \left[ \left( -a + \frac{b}{2}(1+q_R) \left( 1 + \log \frac{a}{b} + \log \frac{2}{1+q_R} \right) \right) \frac{1}{1-q_R} \log \frac{N}{R} \right].
\end{align*}
We notice that
\[
q \in [0,1] \mapsto -a + \frac{b}{2}(1+q) \left( 1 + \log \frac{a}{b} + \log \frac{2}{1+q} \right) 
\]
increases on $[0, 2a/b-1]$, hits $0$ at $2a/b-1$ and decreases on $[2a/b-1,1]$. If $b$ is close enough to $a$, for all $R \geq N^{1/2-a/4}$,
\[
q_R \leq 1/2 + a/4 < 2a/b -1.
\]
We deduce that if $q_R < q_{ab}$,
\[
e^{-a\log N/(1-q_R) } q_R^{\frac{b}{2} \frac{1+q_R}{1-q_R} \log \frac{N}{R}} \sum_{i = 0}^{\frac{b}{2} \frac{1+q_R}{1-q_R} \log \frac{N}{R}-1} \frac{1}{i!} \left( \frac{a\log N}{1-q_R} \right)^i
\leq
N^{-a} \left( \frac{N}{R} \right)^{-\eta'}
\]
for some $\eta'' = \eta''(a,b)$.
Finally, we use \eqref{eq:app_elem1} to bound
\begin{align*}
& e^{-a\log N/(1-q_R) }
\sum_{i \geq \frac{b}{2} \frac{1+q_R}{1-q_R} \log \frac{N}{R}} \frac{1}{i!} \left( q_R \frac{a\log N}{1-q_R} \right)^i \left( 1 + O \left( \frac{1}{\log N} \right) \right)^i \\
& \leq e^{-a\log N/(1-q_R) } \left( \frac{2ae}{b(1+q_R)}  \right)^{\frac{b}{2} \frac{1+q_R}{1-q_R} \log \frac{N}{R}}
\end{align*}
which is smaller than $N^{-a} (N/R)^{-\eta''}$
according to the previous estimates. Putting things together, we have obtained
\begin{align*}
& \prob_{Nx_0}^{D_N} \left( \ell_x^{\tau_{\partial D_N}} \geq g a \log^2 N, A_N(x \to R) > \frac{b}{2} \frac{1+q_R}{1-q_R} \log \frac{N}{R} \right)
\leq \frac{G^{D_N}(Nx_0,x)}{G^{D_N}(x,x)} N^{-a} \left( \frac{N}{R} \right)^{-\eta'\wedge \eta''}.
\end{align*}
Coming back to \eqref{eq:app_2}, it shows that
\[
\PROB{Nx_0}{ \ell_x^{\tau_{\partial D_N}} \geq g a \log^2N, G_N^{b,\eps}(x)^c } \leq  p(\eps) \frac{G^{D_N}(Nx_0,x)}{G^{D_N}(x,x)} N^{-a}
\]
for some $p(\eps)>0$ depending on $a,b,\eps$ going to $0$ when $\eps \to 0$. This concludes the proof.
\end{proof}

\begin{proof}[Proof of Lemma \ref{lem:app_second_moment}]
We have
\begin{equation}
\label{eq:app_4}
\Expect{ \bar{\mu}_{x_0;N}^{D,a}(\C)^2 } =
\frac{\log^2 N}{N^{4-2a}} \sum_{x,y \in \Z^2} \indic{x/N, y/N \in D^\eps} \PROB{Nx_0}{ \ell_x^{\tau_{\partial D_N}}, \ell_y^{\tau_{\partial D_N}} \geq g a \log^2 N, G_N^{b,\eps}(x), G_N^{b,\eps}(y) }.
\end{equation}
The contribution to the above sum of points $x,y$ satisfying $\abs{x-y} \leq N^{1/2-a/4}$ goes to zero. Indeed, thanks to the first moment estimate of Property \ref{prop:tightness}, it is at most
\[
\frac{\log^2 N}{N^{4-2a}} N^{1-a/2} \sum_{x \in \Z^2} \PROB{Nx_0}{\ell_x^{\tau_{\partial D_N}} \geq ga \log^2 N} = \frac{\log N}{N^{1-a/2}} \Expect{\mu^{D_N}_{x_0}(\C) } \leq C \frac{\log N}{N^{1-a/2}}
\]
which goes to zero since $a < 2$.
We now take $x,y \in \Z^2$ such that $x/N,y/N \in D^\eps$ and $\abs{x-y} > N^{1/2-a/4}$. The goal is to bound the probability written in \eqref{eq:app_4}. Take $R \in (2^p)_{p \geq 1} \cap [N^{1/2-a/4}, \eps N]$ so that
\[
c \abs{x-y} \leq R \leq \abs{x-y}/10
\]
with $c>0$ which may depend on $\eps$ and on the domain $D$. Notice that with this choice of $R$ and because $\abs{x-y} > N^{1/2-a/4}$, the quantity $q_R$ defined in \eqref{eq:app_def_qR} stays bounded away from $1$. Now, the probability in \eqref{eq:app_4} is at most
\begin{equation}
\label{eq:app_5}
\PROB{Nx_0}{\ell_x^{\tau_{\partial D_N}}, \ell_y^{\tau_{\partial D_N}} \geq g a \log^2 N, A_N(x\to R), A_N(y\to R) \leq \frac{b}{2} \frac{1+q_R}{1-q_R} \log \frac{N}{R} }.
\end{equation}
We described the joint law of $(\ell_x^{\tau_{\partial D_N}},\ell_y^{\tau_{\partial D_N}}, A_N(x \to R), A_N(y\to R))$ in the discussion following Lemma \ref{lem:app_hitting_proba}. With the notations therein and with \eqref{eq:app_10}, the probability \eqref{eq:app_5} is equal to
\begin{align*}
& \sum_{\substack{1 \leq k_x \leq \frac{b}{2} \frac{1+q_R}{1-q_R} \log \frac{N}{R}\\1 \leq k_y \leq \frac{b}{2} \frac{1+q_R}{1-q_R} \log \frac{N}{R}}} \PROB{Nx_0}{A_N(x\to R) = k_x, A_N(y\to R) = k_y} \\
& ~~~~~~~~~~~~~~~\times \Prob{\Gamma(k_x,1), \Gamma(k_y,1) \geq \left( 1 + O \left( \frac{1}{\log N} \right) \right) a \log N /(1-q_R)} \\
& \leq O(1) e^{-2a\log N /(1-q_R)} \sum_{2 \leq k \leq b \frac{1+q_R}{1-q_R} \log \frac{N}{R}} \PROB{Nx_0}{A_N(x\to R) + A_N(y\to R) = k} \\
& ~~~~~~~~~~~~~~~\times \sum_{i=0}^{k-2} \frac{1}{i!} \left( \frac{2a\log N}{1-q_R} \right)^i
\end{align*}
With \eqref{eq:app_11}, we get that the probability \eqref{eq:app_5} is at most
\begin{align*}
& O(1) e^{-2a\log N /(1-q_R)} \PROB{Nx_0}{\tau_x \wedge \tau_y < \tau_{\partial D_N}} \\
& \times \sum_{2 \leq k \leq b \frac{1+q_R}{1-q_R} \log \frac{N}{R}} 
\left( \frac{2q_R}{1+q_R} \right)^{k-1}
\sum_{i=0}^{k-2} \frac{1}{i!} \left( \frac{2a\log N}{1-q_R} \right)^i \\
& = O(1) e^{-2a\log N /(1-q_R)} \PROB{Nx_0}{\tau_x \wedge \tau_y < \tau_{\partial D_N}} \frac{q_R}{1-q_R} \sum_{i=0}^{b \frac{1+q_R}{1-q_R} \log \frac{N}{R} -2} \frac{1}{i!} \left( \frac{4a q_R\log N}{1-q_R^2} \right)^i
\end{align*}
by exchanging the two sums. We now use \eqref{eq:app_elem2} with
\[
\mu = \frac{4aq_R\log N}{1-q_R^2} \frac{1-q_R}{b(1+q_R) \log (N/R)} = \frac{4a}{b} \frac{1}{(1+q_R)^2}
\]
which is bigger than $1$ if $b$ is close enough to $a$ (recall that $q_R$ stays bounded away from 1). We obtain that the probability \eqref{eq:app_5} is at most
\begin{align*}
& O(1) \PROB{Nx_0}{\tau_x \wedge \tau_y < \tau_{\partial D_N}} q_R  e^{-2a\log N /(1-q_R)} \left( e \frac{4a}{b} \frac{1}{(1+q_R)^2} \right)^{b \frac{1+q_R}{1-q_R} \log \frac{N}{R}} \\
& = O(1) \PROB{Nx_0}{\tau_x \wedge \tau_y < \tau_{\partial D_N}} q_R N^{-2a} \\
& \times \exp \left[ \left( -2a + b (1+q_R) \left( 1 + \log \frac{a}{b} + 2 \log \frac{2}{1+q_R} \right) \right) \frac{1}{1-q_R} \log \frac{N}{R} \right] \\
& \leq O(1) \PROB{Nx_0}{\tau_x \wedge \tau_y < \tau_{\partial D_N}} q_R N^{-2a} \exp \left[ \left( -2a + b(1+q_R) \left( \frac{a}{b} + 2 \frac{1-q_R}{1+q_R} \right) \right) \frac{1}{1-q_R} \log \frac{N}{R} \right] \\
& = O(1) \PROB{Nx_0}{\tau_x \wedge \tau_y < \tau_{\partial D_N}} q_R N^{-2a} \left( \frac{N}{R} \right)^{2b-a}.
\end{align*}
To wrap things up, we have obtained
\begin{align*}
& \frac{\log^2N}{N^{4-2a}} \sum_{x,y \in \Z^2} \indic{x/N, y/N \in D^\eps, \abs{x-y} \geq N^{1/2-a/4}} \PROB{Nx_0}{\ell_x^{\tau_{\partial D_N}}, \ell_y^{\tau_{\partial D_N}} \geq ga \log^2N, G_N^{b,\eps}(x), G_N^{b,\eps}(y) } \\
& \leq \frac{O(1)}{N^4} \sum_{x,y \in \Z^2} \indic{x/N, y/N \in D^\eps, \abs{x-y} \geq N^{1/2-a/4}} \log \frac{N}{\abs{x-Nx_0}} \log \frac{N}{\abs{x-y}} \left( \frac{N}{\abs{x-y}} \right)^{2b-a}
\end{align*}
which is bounded uniformly in $N$ if $b$ is chosen close enough to $a$ so that $2b-a < 2$. The energy estimate \eqref{eq:lem_app_energy} follows as well. This finishes to prove Lemma \ref{lem:app_second_moment}.
\end{proof}

\section{Joint convergence of measures and trajectories}\label{sec:joint}

In this section, we state a natural extension of Theorem \ref{th:convergence} that follows from our approach. Theorem \ref{th:extension} below extends Theorem \ref{th:convergence} in two directions. It considers the joint convergence of the measure together with the associated random walk and it considers finitely many independent random walk trajectories. This generalisation plays a crucial role in the paper \cite{ABJL21} which studies a multiplicative chaos associated to Brownian loop soup.

Let $\Dc \Xc \Zc = \{ (D^i, x_i, z_i), i = 1 \dots r \} \in \Sc$ be a collection of domains with starting points and ending points. Let $X^{(i)} = (X^{(i)}_t, 0 \leq t \leq \tau^i), i =1 \dots r$, be $r$ independent random walks distributed according to $\P_{Nx_i, Nz_i}^{D_N^i}$ and, for any $\Dc'\Xc'\Zc' = \{(D^i,x_i,z_i), i \in I\} \subset \Dc\Xc\Zc$, recall the definition \eqref{eq:def_multipoint_discrete} of the measure $\mu_{\Xc',\Zc';N}^{\Dc',a}$ encoding the set of $a$-thick points coming from the interaction of the random walks $X^{(i)}$, $i \in I$. We rescale the walk $X^{(i)}$ in time and space and define $X_N^{(i)} = ( N^{-1} X^{(i)}_{N^2 t}, 0 \leq t \leq \tau^i / N^2 )$.

To give a precise meaning of the convergence of the above random walks towards Brownian motion, we need to define a topology on the set $\Pc$ of càdlàg paths in $\R^2$ with finite durations. If $(\wp^1_t, 0\leq t \leq T^1)$ and $(\wp^2_t, 0\leq t \leq T^2)$ are two such paths, we define the distance
\[
d( \wp^1, \wp^2 ) := | \log (T^1 / T^2) | + d_\text{Sk}( (\wp^1_{t T^1}, 0 \leq t \leq 1), (\wp^2_{t T^2}, 0 \leq t \leq 1) )
\]
where $d_\text{Sk}$ denotes the Skorokhod distance between càdlàg functions defined on $[0,1]$ with values in $\R^2$ (see e.g. Section 12 of \cite{billingsley}). We equip the set $\Pc$ with the topology associated to that distance.

Finally, for any Borel set $U \subset \R^2$, we will denote by $\mathfrak{M}(U)$ the set of Borel measures on $U$ equipped with the topology of vague convergence on $U$.

\begin{theorem}\label{th:extension}
For any $\Dc \Xc \Zc \in \Sc$,
\[
\left( \mu_{\Xc',\Zc';N}^{\Dc',a}, \Dc'\Xc'\Zc' \subset \Dc \Xc \Zc, X_N^{(i)}, i = 1 \dots r \right) \in \prod_{\Dc' \subset \Dc} \mathfrak{M} \left( \bigcap_{D' \in \Dc'} D' \right) \times \Pc^r
\]
converges weakly relative to the product topology to
\[
\left( e^{c_0 a/g} \Mc_{\Xc',\Zc'}^{\Dc',a}, \Dc'\Xc'\Zc' \subset \Dc \Xc \Zc, B^{(i)}, i = 1 \dots r \right),
\]
where $B^{(i)}$, $i = 1\dots r$, are independent Brownian paths distributed according to $\P_{x_i, z_i}^{D^i}$, $i=1 \dots r$, and $\Mc_{\Xc',\Zc'}^{\Dc',a}$, $\Dc'\Xc'\Zc'$, are the multipoint Brownian chaos associated to $B^{(i)}$, $i=1 \dots r$, defined in Section \ref{subsec:Brownian_multiplicative_chaos}.
\end{theorem}

We now explain the slight modifications needed in order to prove Theorem \ref{th:extension}.
Firstly, extending the convergence of Theorem \ref{th:extension} to the case of finitely many trajectories does not require any modification. Indeed, Proposition \ref{prop:intersection} shows tightness of the sequence 
$\left( \mu_{\Xc',\Zc';N}^{\Dc',a}, \Dc'\Xc'\Zc' \subset \Dc \Xc \Zc \right)$, $N \geq 1$.
Let $\left( \mu_{\Xc',\Zc'}^{\Dc',a}, \Dc'\Xc'\Zc' \subset \Dc \Xc \Zc \right)$ be any subsequential limit. By Lemma \ref{lem:uncountable_extraction}, we can extract a further subsequence and we obtain an uncountable family $\left( \mu_{\Xc',\Zc'}^{\Dc',a}, \Dc'\Xc'\Zc' \in \Sc \right)$ of measures. As shown in the proof of Theorem \ref{th:convergence}, this family satisfies Properties \ref{charac1} - \ref{charac3} (up to the multiplicative factor $e^{c_0 a/g}$) which characterise the law of $\left( \Mc_{\Xc',\Zc'}^{\Dc',a}, \Dc'\Xc'\Zc' \in \Sc \right)$ by Theorem \ref{th:charac}.

It remains to explain how to deal with the joint convergence of the measures together with the underlying random walks. We proceed again by first showing tightness and then identifying the law of the subsequential limits. Tightness is clear since each component converges (we have already seen that the measures converge, and the random walks converge by Donsker invariance principle).
The study of the law of the subsequential limit is then very similar to what we have done, as soon as we have an appropriate generalisation of Theorem \ref{th:charac} that we explain below in details.

This time, we want to characterise the law of
$( (\Mc_{\Xc, \Zc}^{\Dc, a}, B_{\Xc, \Zc}^\Dc), \Dc \Xc \Zc \in \Sc )$, where for any $\Dc \Xc \Zc = \{ (D^i, x_i, z_i), i =1 \dots r \} \in \Sc$, we denoted by $B_{\Xc, \Zc}^\Dc$ the collection $(B_{x_i, z_i}^{D^i}, i =1 \dots r)$ of independent Brownian trajectories associated to the measures.
Consider a stochastic process
\begin{equation}
\label{eq:stoc_pro}
\Dc \Xc \Zc \in \Sc \mapsto (\mu_{\Xc, \Zc}^{\Dc, a}, B_{\Xc, \Zc}^\Dc) \in \mathfrak{M} \left( \bigcap_{D \in \Dc} D \right) \times \Pc^{\# \Dc \Xc \Zc}
\end{equation}
and the following properties:

\begin{enumerate}[label=(\subscript{P}{{\arabic*}}')]
\item(Average value)
\label{charac1'}
For all $\Dc\Xc\Zc = \left\{ (D_i,x_i,z_i), i=1 \dots r \right\} \in \Sc$ and for all Borel set $A \subset \C$,
\begin{align*}
& \Expect{ \mu^{\Dc,a}_{\Xc,\Zc}(A) }
= \int_A dx \int_{\mathsf{a} \in E(a,r)} d \mathsf{a} \prod_{k=1}^r \psi_{x_k,z_k}^{D_k,a_k}(x).
\end{align*}
\item(Markov property)
\label{charac2'}
Let $\Dc\Xc\Zc \in \Sc$, $(D,x_0,z) \in \Dc\Xc\Zc$ and let $D'$ be a nice subset of $D$ containing $x_0$. Let $Y$ be distributed according to $B_{\tau_{\partial D'}}$ under $\P_{x_0,z}^D$. The joint law of $((\mu_{\Xc',\Zc'}^{\Dc',a}, B_{\Xc',\Zc'}^{\Dc'}), \Dc'\Xc'\Zc' \subset \Dc\Xc\Zc)$ is the same as the joint law given by for all $\Dc'\Xc'\Zc' \subset \Dc\Xc\Zc$,
\[
\left\{
\begin{array}{l}
(\mu_{\Xc',\Zc'}^{\Dc',a}, B_{\Xc',\Zc'}^{\Dc'})
\mathrm{~if~} (D,x_0,z) \notin \Dc'\Xc'\Zc', \\
(\mu_{\bar{\Dc}\bar{\Xc}\bar{\Zc} \cup \{(D',x_0,Y)\} }^a
+ \mu_{\bar{\Dc}\bar{\Xc}\bar{\Zc} \cup \{(D,Y,z)\} }^a
+ \mu_{\bar{\Dc}\bar{\Xc}\bar{\Zc} \cup \{(D',x_0,Y), (D,Y,z)\} }^a , \tilde{B}_{\Xc',\Zc'}^{\Dc'} )
\mathrm{~otherwise},
\end{array}
\right.
\]
where in the second line we denote $\bar{\Dc}\bar{\Xc}\bar{\Zc} = \Dc'\Xc'\Zc' \backslash \{(D,x_0,z)\}$ and $\tilde{B}_{\Xc',\Zc'}^{\Dc'}$ is the collection of trajectories obtained from $B_{\Xc',\Zc'}^{\Dc'}$ as follows. For all $(\bar{D}, \bar{x}_0, \bar{z}) \in \bar{\Dc}\bar{\Xc}\bar{\Zc}$, $B_{\bar{x}_0,\bar{z}}^{\bar{D}}$ is unchanged. $B_{x_0,z}^D$ is replaced by the concatenation of $B_{x_0,Y}^{D'}$ and $B_{Y,z}^D$.
\item(Independence)
\label{charac2bis'}
For all disjoint sets $\Dc\Xc\Zc, \Dc'\Xc'\Zc' \in \Sc$, $(\mu^{\Dc,a}_{\Xc,\Zc}, B^\Dc_{\Xc,\Zc})$ and $(\mu_{\Xc',\Zc'}^{\Dc',a}, B_{\Xc',\Zc'}^{\Dc'})$ are independent.
\item(Non-atomicity)
\label{charac3'}
For all $\Dc\Xc\Zc \in \Sc$, with probability one, simultaneously for all $x \in \C$, $\mu^{\Dc,a}_{\Xc,\Zc}(\{x\}) = 0$.
\item
\label{charac4'}
For all $\{ (D,x_0,z) \} \in \Sc$, $B_{x_0, z}^D \sim \P_{x_0, z}^D$.
\end{enumerate}

\begin{theorem}\label{th:charac'}
The process
$\left ((\Mc^{\Dc,a}_{\Xc,\Zc}, B_{\Xc, \Zc}^\Dc), \Dc\Xc\Zc \in \Sc \right)$
from Section \ref{subsec:Brownian_multiplicative_chaos} satisfies Properties \ref{charac1'}-\ref{charac4'}. Moreover, if
$\left ((\mu^{\Dc,a}_{\Xc,\Zc}, B_{\Xc, \Zc}^\Dc), \Dc\Xc\Zc \in \Sc \right)$
is another process with target spaces as in \eqref{eq:stoc_pro} satisfying Properties \ref{charac1'}-\ref{charac4'}, then it has the same law as
$\left ((\Mc^{\Dc,a}_{\Xc,\Zc}, B_{\Xc, \Zc}^\Dc), \Dc\Xc\Zc \in \Sc \right)$.
\end{theorem}

We want to emphasise again that it is crucial that the characterisation does not rely on the measurability of the measures with respect to the Brownian paths.

The proof of Theorem \ref{th:charac'} is similar to the proof of Theorem \ref{th:charac} and we omit it.

\appendix

\section{Multipoint Brownian multiplicative chaos}\label{app:intersection}

This section is devoted to the proof of Propositions \ref{prop:def_measures}, \ref{prop:intersection},  \ref{prop:intersection_multipoint} and \ref{prop:other_charac}.
We start with Proposition \ref{prop:intersection}.

\begin{proof}[Proof of Proposition \ref{prop:intersection}]
We use the notations of Section \ref{subsec:further_results} and for $i=1 \dots r$, we will denote
\[
f^\eps_i(x) := \abs{\log \eps} \eps^{-a_i} \indic{\frac{1}{\eps}L_{x,\eps}^{(i)} \geq 2a_i \abs{\log \eps}^2}.
\]
We recall that
\begin{equation}
\label{eq:proof_th_intersection0}
\lim_{\eps \to 0} \Expect{f^\eps_i(x)} = \psi_{x_i,z_i}^{D_i,a_i}(x)
\end{equation}
and that we can bound
\begin{equation}
\label{eq:proof_th_intersection1}
\sup_{\eps >0} \Expect{f^\eps_i(x)} \leq C \psi_{x_i,z_i}^{D_i,a_i}(x)
\end{equation}
for some $C>0$. See \cite[Proposition 3.1]{jegoGMC}. We moreover recall that for all $\eta >0$, we can decompose
\[
f^\eps_i(x) = \rho^{\eta,\delta;\eps}_i(x) + f^{\eta,\delta;\eps}_i(x)
\]
where
\begin{equation}
\label{eq:proof_prop_intersection8}
\lim_{\delta \to 0} \limsup_{\eps \to 0} \Expect{ \rho^{\eta,\delta;\eps}_i(x) } =0
\end{equation}
and for all $\delta>0$, $x \neq y$,
\begin{equation}
\label{eq:proof_prop_intersection3}
\sup_{\eps >0} \Expect{ f_i^{\eta,\delta;\eps}(x) f_i^{\eta,\delta;\eps}(y) } \leq C_{\eta,\delta} \abs{x-y}^{-a_i-\eta}
\end{equation}
and for all $x \neq y$,
\begin{equation}
\label{eq:proof_prop_intersection4}
\lim_{\eps,\eps' \to 0} \Expect{ (f_i^{\eta,\delta;\eps}(x) - f_i^{\eta,\delta;\eps'}(x)) (f_i^{\eta,\delta;\eps}(y) - f_i^{\eta,\delta;\eps'}(y)) } = 0.
\end{equation}
This follows from the decomposition of the measure using ``good'' and ``bad'' events used in \cite{jegoGMC}. Let us detail this decomposition. Let $\delta >0$ and $b_i>a_i$ be very close to $a_i$ (depending on $\eta$). We introduce the good event (see (21) in \cite{jegoGMC})
\[
G_\eps(x) := \left\{ \forall r \in [\eps,\delta] \cap \{e^{-n}, n \geq 1\} : \frac{1}{r} L_{x,r}^{(i)} \leq 2 b_i |\log r|^2 \right\}
\]
and define
\[
f_i^{\eta,\delta;\eps}(x) := \mathbf{1}_{G_\eps(x)} f_i^\eps(x)
\quad \text{and} \quad
\rho^{\eta,\delta;\eps}_i(x) := (1- \mathbf{1}_{G_\eps(x)}) f_i^\eps(x).
\]
Then \eqref{eq:proof_prop_intersection8} amounts to saying that an $a_i$-thick point is not $b_i$-thick (see \cite[Proposition 3.1]{jegoGMC}), \eqref{eq:proof_prop_intersection3} shows that the measure restricted to good events is bounded in $L^2$ (see (52) of \cite{jegoGMC}) and \eqref{eq:proof_prop_intersection4} is proved in the course of showing that the measure restricted to good events is Cauchy in $L^2$ (see \cite[Proposition 5.1]{jegoGMC}).

We are now going to prove by induction on $r \geq 1$ the claims \textit{(i)}, \textit{(ii)} and \textit{(iv)} of Proposition \ref{prop:intersection} together with the claim that for all $\alpha < 2 - a$ (recall that $a = a_1 + \dots + a_r$), we can decompose
\begin{equation}
\label{eq:proof_prop_intersection5}
\bigcap_{i=1}^r \Mc_{x_i,z_i}^{D_i,a_i} = \rho_\delta + \Mc_\delta
\end{equation}
where $\Expect{\rho_\delta(\C)} \to 0$ as $\delta \to 0$ and for all $\delta >0$
\[
\Expect{ \int_{\C^2} \frac{1}{\abs{x-y}^\alpha} \Mc_\delta(dx) \Mc_\delta(dy) } < \infty.
\]
This latter claim implies \textit{(v)} by Frostman's lemma.
The case $r=1$ follows from \cite{jegoGMC} (in this case, \textit{(ii)} is an empty statement). Let $r \geq 2$ and assume the above results for $r-1$. Let $\alpha, \eta>0$ be such that $a_r < \alpha < \alpha + \eta < 2 - (a_1 + \dots + a_{r-1})$. We can decompose
\[
\bigcap_{i=1}^{r-1} \Mc_{x_i,z_i}^{D_i,a_i} = \rho_\delta + \Mc_\delta
\]
with $\Expect{\rho_\delta(\C)} \to 0$ as $\delta \to 0$ and for all $\delta >0$
\begin{equation}
\label{eq:proof_prop_intersection6}
\Expect{ \int_{\C^2} \frac{1}{\abs{x-y}^{\alpha+\eta}} \Mc_\delta(dx) \Mc_\delta(dy) } < \infty.
\end{equation}
\eqref{eq:proof_prop_intersection3}, \eqref{eq:proof_prop_intersection4} and \eqref{eq:proof_prop_intersection6} show that for all $A \in \Bc(\C)$, and for all $\delta >0$,
\[
\int_A f^{\eta,\delta;\eps}_r(x) \Mc_\delta(dx), \eps >0,
\]
is a Cauchy sequence in $L^2$. This defines a limiting measure $\tilde{\Mc}_\delta$ which satisfies by Fatou's lemma and \eqref{eq:proof_prop_intersection3}
\[
\Expect{ \int_{\C^2} \frac{1}{\abs{x-y}^{\alpha-a_r}} \tilde{\Mc}_\delta(dx) \tilde{\Mc}_\delta(dy) } < \infty.
\]
Moreover,
\[
\lim_{ \delta \to 0} \limsup_{\eps \to 0} \Expect{ \int_\C f^\eps_r(x) \bigcap_{i=1}^{r-1} \Mc_{x_i,z_i}^{D_i,a_i}(dx) - \int_{\C} f_r^{\eta,\delta;\eps}(x) \Mc_\delta(dx) } = 0.
\]
This shows that for all $A \in \Bc(\C)$,
$
\int_A f^\eps_r(x) \bigcap_{i=1}^{r-1} \Mc_{x_i,z_i}^{D_i,a_i}(dx)
$
converges in $L^1$ as $\eps \to 0$ and also shows that the limiting measure can be decomposed as expected in \eqref{eq:proof_prop_intersection5}. 
This concludes the proof of the convergence of the measure \eqref{eq:prop_intersection}. This also shows that for all $A \in \Bc(\C)$, we can exchange the expectation and the limit:
\begin{align*}
\Expect{ \lim_{\eps \to 0} \int_A f^\eps_r(x) \bigcap_{i=1}^{r-1} \Mc_{x_i,z_i}^{D_i,a_i}(dx) }
& = \lim_{\eps \to 0} \int_A \Expect{ f^\eps_r(x) } \Expect{ \bigcap_{i=1}^{r-1} \Mc_{x_i,z_i}^{D_i,a_i}(dx) }.
\end{align*}
Now, using \textit{(iv)} for $r-1$, we see that $\Expect{ \bigcap_{i=1}^{r-1} \Mc_{x_i,z_i}^{D_i,a_i}(dx) } = \prod_{i=1}^{r-1} \psi_{x_i,z_i}^{D_i,a_i}(x) dx$ and then by dominated convergence theorem and \eqref{eq:proof_th_intersection0} and \eqref{eq:proof_th_intersection1}, we obtain that
\begin{align*}
\Expect{ \lim_{\eps \to 0} \int_A f^\eps_r(x) \bigcap_{i=1}^{r-1} \Mc_{x_i,z_i}^{D_i,a_i}(dx) }
& = \lim_{\eps \to 0} \int_A \Expect{ f^\eps_r(x) } \prod_{i=1}^{r-1} \psi_{x_i,z_i}^{D_i,a_i}(x) dx \\
& = \int_A \prod_{i=1}^r \psi_{x_i,z_i}^{D_i,a_i}(x) dx.
\end{align*}

We are now going to show that 
$\bigcap_{i=1}^r \Mc_{x_i,z_i;\eps}^{D_i,a_i}$
converges to the same limiting measure as \eqref{eq:prop_intersection}. For this purpose, it is enough to show that for all $A \in \Bc(\C)$,
\[
\Expect{ \abs{ \int_A \prod_{i=1}^r f^\eps_i(x) dx - \int_A f^\eps_r(x) \bigcap_{i=1}^{r-1} \Mc_{x_i,z_i}^{D_i,a_i}(dx) } }
\]
tends to zero as $\eps \to 0$.
For each $i=1 \dots r$, consider the decomposition
\[
f^\eps_i(x) = \rho^{\eta,\delta;\eps}_i(x) + f^{\eta,\delta;\eps}_i(x)
\]
with $\eta>0$ in \eqref{eq:proof_prop_intersection3} chosen so that $r \eta + a_1 + \dots a_r < 2$. For $\eps',\delta >0$, we can bound
\begin{align}\label{eq:proof_prop_intersection7}
& \Expect{ \abs{ \int_A \prod_{i=1}^r f^\eps_i(x) dx - \int_A f^\eps_r(x) \bigcap_{i=1}^{r-1} \Mc_{x_i,z_i}^{D_i,a_i}(dx) } } \\
& \leq \Expect{ \abs{ \int_A f^\eps_r(x) \left( \prod_{i=1}^{r-1} f^\eps_i(x) - \prod_{i=1}^{r-1} f^{\eps'}_i(x) \right) dx}} \nonumber\\
& + \Expect{ \abs{ \int_A f^\eps_r(x) \prod_{i=1}^{r-1} f^{\eps'}_i(x) dx - \int_A f^\eps_r(x) \bigcap_{i=1}^{r-1} \Mc_{x_i,z_i}^{D_i,a_i}(dx) } }.\nonumber
\end{align}
By the case $r-1$, the second right hand side term tends to zero as $\eps' \to 0$. By writing for $i,k=1 \dots r-1$, $\eps_i^k = \eps'$ if $i \leq k-1$ and $\eps_i^k = \eps$ otherwise, we can write
\[
\prod_{i=1}^{r-1} f^\eps_i(x) - \prod_{i=1}^{r-1} f^{\eps'}_i(x) = \sum_{k=1}^{r-1} (f^{\eps}_k(x) - f^{\eps'}_k(x)) \prod_{\substack{1 \leq i \leq r-1\\i \neq k}} f^{\eps_i^k}_i(x).
\]
By triangle inequality, to bound the first right hand side term of \eqref{eq:proof_prop_intersection7}, it is thus enough to bound
\[
\Expect{ \abs{ \int_A (f_1^\eps(x) - f_1^{\eps'}(x)) \prod_{i=2}^r f_i^\eps(x) dx } }
\]
and $r-2$ other very similar terms. This is at most
\begin{align*}
& \Expect{ \abs{ \int_A (f_1^{\eta,\delta;\eps}(x) - f_1^{\eta,\delta;\eps'}(x)) \prod_{i=2}^r f_i^{\eta,\delta;\eps}(x) dx } } \\
& + \Expect{ \abs{ \int_A (f_1^\eps(x) - f_1^{\eps'}(x)) \prod_{i=2}^r f_i^\eps(x) dx - \int_A (f_1^{\eta,\delta;\eps}(x) - f_1^{\eta,\delta;\eps'}(x)) \prod_{i=2}^r f_i^{\eta,\delta;\eps}(x) dx } }.
\end{align*}
By independence and because for all $i=1 \dots r$, $x \in \C$, $\lim_{\delta \to 0} \limsup_{\eps \to 0} \Expect{\rho_i^{\eta,\delta;\eps}(x)} = 0$, dominated convergence theorem (the domination is provided by \eqref{eq:proof_th_intersection1}) shows that the $\limsup_{\eps,\eps' \to 0}$ of the second right hand side term goes to zero as $\delta \to 0$. By Cauchy-Schwarz and \eqref{eq:proof_prop_intersection3}, the first term is at most
\[
C_{\eta,\delta} \int_{A \times A} \frac{1}{|x-y|^{a_2 + \dots + a_r+ (r-1)\eta}} \Expect{ (f_1^{\eta,\delta;\eps}(x) - f_1^{\eta,\delta;\eps'}(x)) (f_1^{\eta,\delta;\eps}(y) - f_1^{\eta,\delta;\eps'}(y)) } dx dy
\]
which tends to zero as $\eps, \eps' \to 0$ by \eqref{eq:proof_prop_intersection3}, \eqref{eq:proof_prop_intersection4} and dominated convergence theorem (note that we obtain an integrable domination because $r \eta + a_1 + \dots + a_r < 2$). This concludes the induction proof of \textit{(i)}, \textit{(ii)} and \textit{(iv)}.

Finally, by induction on $r$, the measurability statement \textit{(iii)} follows from \textit{(ii)}. To conclude the proof, it remains to check \textit{(vi)}. The measurability of the process is clear at the level of approximation, i.e. for all $\eps >0$,
\[
(a_i)_{i = 1 \dots r} \in \{ (\alpha_i)_{i =1 \dots r} \in (0,2)^r: \sum \alpha_i < 2 \} \mapsto \bigcap_{i=1}^r \Mc_{x_i,z_i,\eps}^{D_i,a_i}
\]
is a measurable process (with appropriate topology). The claim follows from \textit{(i)} since a pointwise limit of measurable functions is measurable.
\end{proof}

\begin{remark}\label{rem:intersection}
We now address a remark which will be useful for the proof of Proposition \ref{prop:intersection_multipoint}. The proof of Proposition \ref{prop:intersection} actually shows that for all Borel set $A \subset \C$,
$
\bigcap_{i=1}^r \Mc_{x_i,z_i;\eps}^{D_i,a_i}(A)
$
converges in $L^1$ as $\eps \to 0$ towards what we denoted $\bigcap_{i=1}^r \Mc_{x_i,z_i}^{D_i,a_i}(A)$. \cite{jegoGMC} considered not only the spatial configuration of the thick points but also the deviation of the local times by looking at the measure: for all $A \in \Bc(\C)$ and $T \in \Bc(\R \cup \{+\infty\})$,
\[
\bar{\Mc}_{x_1,z_1;\eps}^{D_1,a_1}(A \times T) := \abs{\log \eps} \eps^{-a_1} \int_A \indic{ \sqrt{\frac{1}{\eps} L_{x,\eps}^{(1)}} - \sqrt{2a_1} \abs{\log \eps} \in T} dx.
\]
\cite[Proposition 6.1]{jegoGMC} shows that for all Borel sets $A \subset \C$ and $T \subset \R$ with $\inf T>-\infty$, $\bar{\Mc}_{x_1,z_1;\eps}^{D_1,a_1}(A \times T)$ converges in $L^1$ as $\eps \to 0$ towards 
\[
\Mc_{x_1,z_1;\eps}^{D_1,a_1}(A) \int_T \frac{1}{\sqrt{2a_1}} e^{-\sqrt{2a_1} t} dt.
\]
Using this result, a straightforward extension of Proposition \ref{prop:intersection} shows similarly that for all Borel sets $A \subset \C$ and $T_i \subset \R$ with $\inf T_i > - \infty$, the following convergence in $L^1$ holds
\[
\abs{\log \eps}^r \eps^{-a} \int_A \prod_{i=1}^r \indic{ \sqrt{\frac{1}{\eps} L^{(i)}_{x,\eps}} - \sqrt{2 a_i} \abs{\log \eps} \in T_i} dx
\xrightarrow[\eps \to 0]{} \bigcap_{i=1}^r \Mc_{x_i,z_i}^{D_i,a_i}(A) \prod_{i=1}^r \int_{T_i} \frac{1}{\sqrt{2a_i}} e^{-\sqrt{2a_i}t_i} dt_i.
\]
\end{remark}

We are now ready to prove Propositions \ref{prop:def_measures} and \ref{prop:intersection_multipoint}.

\begin{proof}[Proof of Propositions \ref{prop:def_measures} and \ref{prop:intersection_multipoint}]
To ease the notations, we will restrict ourselves to the case $r=2$. The general case $r \geq 2$ follows along the same lines. Let $A \subset \C$ be a Borel set.
Let $\eta >0$. We have
\begin{align*}
& \abs{\log \eps} \eps^a \int_A dx \indic{L_{x,\eps}^{(1)} + L_{x,\eps}^{(2)} \geq 2a \eps \abs{\log \eps}^2, L_{x,\eps}^{(1)}>0, L_{x,\eps}^{(2)}>0 } \\
& \leq \sum_{\alpha \in \frac{\eta}{\abs{\log \eps}} \N} \abs{\log \eps} \eps^a \int_A dx \indic{ \frac{1}{\eps} L_{x,\eps}^{(1)} \in 2 \left(\alpha, \alpha + \frac{\eta}{\abs{\log \eps}} \right] \abs{\log \eps}^2 , \frac{1}{\eps} L_{x,\eps}^{(2)} >2 (a-\alpha) \abs{\log \eps}^2 - g\eta \abs{\log \eps} } \\
& = \frac{1}{\eta} \int_0^\infty d \alpha \abs{\log \eps}^2 \eps^a \int_A dx \indic{ \frac{1}{\eps} L_{x,\eps}^{(1)} \in 2 \left(\alpha_\eps, \alpha_\eps + \frac{\eta}{\abs{\log \eps}} \right] \abs{\log \eps}^2 , \frac{1}{\eps} L_{x,\eps}^{(2)} >2 (a-\alpha_\eps) \abs{\log \eps}^2 - g\eta \abs{\log \eps} }
\end{align*}
where $\alpha_\eps = \frac{\eta}{\abs{\log \eps}} \floor{ \frac{\abs{\log \eps}}{\eta} \alpha}$. Let $K$ be a large integer. We now have
\begin{align*}
& \abs{\log \eps} \eps^a \int_A dx \indic{L_{x,\eps}^{(1)} + L_{x,\eps}^{(2)} \geq 2a \eps \abs{\log \eps}^2, L_{x,\eps}^{(1)}>0, L_{x,\eps}^{(2)}>0 } \\
& \leq \frac{1}{\eta} \int_0^\infty d\alpha \sum_{k=0}^{K-1} \indic{\alpha \in \alpha_\eps + \frac{\eta}{\abs{\log \eps}} \left[ \frac{k}{K}, \frac{k+1}{K} \right)} \\
& \times \abs{\log \eps}^2 \eps^a \int_A dx \indic{ \frac{1}{\eps} L_{x,\eps}^{(1)} \in 2 \left(\alpha_\eps, \alpha_\eps + \frac{\eta}{\abs{\log \eps}} \right] \abs{\log \eps}^2 , \frac{1}{\eps} L_{x,\eps}^{(2)} >2 (a-\alpha_\eps) \abs{\log \eps}^2 - 2\eta \abs{\log \eps} } \\
& \leq \frac{1}{\eta} \int_0^\infty d\alpha \sum_{k=0}^{K-1} \indic{\alpha \in \alpha_\eps + \frac{\eta}{\abs{\log \eps}} \left[ \frac{k}{K}, \frac{k+1}{K} \right)} \abs{\log \eps}^2 \eps^a \\
& \times \int_A dx \indic{ \frac{1}{\eps} L_{x,\eps}^{(1)} \in 2 \left(\alpha - \frac{\eta}{|\log \eps|} \frac{k+1}{K}, \alpha + \frac{\eta}{|\log \eps|} \left( 1 - \frac{k}{K} \right) \right] |\log \eps|^2, \frac{1}{\eps} L_{x,\eps}^{(2)} >2 (a-\alpha - \frac{\eta}{|\log \eps|} \frac{k}{K}) |\log \eps|^2 - 2\eta |\log \eps| }.
\end{align*}
For each $\alpha \in (0,a)$ and $k \in \{0,\dots, K-1\}$, $\abs{\log \eps}^2 \eps^a$ times
\[
\int_A dx \indic{ \frac{1}{\eps} L_{x,\eps}^{(1)} \in 2 \left(\alpha - \frac{\eta}{|\log \eps|} \frac{k+1}{K}, \alpha + \frac{\eta}{|\log \eps|} \left( 1 - \frac{k}{K} \right) \right] |\log \eps|^2, \frac{1}{\eps} L_{x,\eps}^{(2)} >2 (a-\alpha - \frac{\eta}{|\log \eps|} \frac{k}{K}) |\log \eps|^2 - 2\eta |\log \eps| }
\]
converges in $L^1$ towards (see Remark \ref{rem:intersection})
\begin{align*}
& \Mc_{x_0,z_0}^{D_0,\alpha} \cap \Mc_{x_1,z_1}^{D_1,a-\alpha}(A) e^{\eta(1+k/K)} \int_{-\eta(k+1)/(K \sqrt{2\alpha})}^{\eta(1-k/(K\sqrt{2\alpha}))} e^{-\sqrt{2\alpha} t} dt \\
& = \eta (1+o(1)) \Mc_{x_0,z_0}^{D_0,\alpha} \cap \Mc_{x_1,z_1}^{D_1,a-\alpha}(A)
\end{align*}
where $o(1)$ is independent of $\alpha$ and $k$ and goes to zero as $\eta \to 0$ and then $K \to \infty$.
Hence for all $\alpha \in (0,a)$,
\begin{align*}
& \frac{1}{\eta} \sum_{k=0}^{K-1} \indic{\alpha \in \alpha_\eps + \frac{\eta}{\abs{\log \eps}} \left[ \frac{k}{K}, \frac{k+1}{K} \right)} \abs{\log \eps}^2 \eps^a \\
& \times \int_A dx \indic{ \frac{1}{\eps} L_{x,\eps}^{(1)} \in 2\left(\alpha - \frac{\eta}{|\log \eps|} \frac{k+1}{K}, \alpha + \frac{\eta}{|\log \eps|} \left( 1 - \frac{k}{K} \right) \right] |\log \eps|^2, \frac{1}{\eps} L_{x,\eps}^{(2)} >2 (a-\alpha - \frac{\eta}{|\log \eps|} \frac{k}{K}) |\log \eps|^2 - 2\eta |\log \eps| }
\end{align*}
converges in $L^1$ towards
\[
(1+o(1)) \Mc_{x_0,z_0}^{D_0,\alpha} \cap \Mc_{x_1,z_1}^{D_1,a-\alpha}(A).
\]
If $\alpha > a$, the above term converges in $L^1$ to zero.
We are going to conclude with the following elementary reasoning. If $\alpha \in (0,\infty) \mapsto X_\eps^\alpha, \eps >0,$ are random processes almost surely measurable and defined on the same probability space satisfying: for all $\alpha >0$, $(X_\eps^\alpha, \eps >0)$ converges in $L^1$ and $\int_0^\infty \sup_\eps \Expect{\abs{X_\eps^\alpha}} d \alpha < \infty$; then $(\int_0^\infty X_\eps^\alpha d\alpha, \eps >0),$ converges in $L^1$. Indeed
\[
\limsup_{\eps, \eps' \to 0} \Expect{ \abs{ \int_0^\infty X_\eps^\alpha d\alpha - \int_0^\infty X_{\eps'}^\alpha d\alpha } } \leq \limsup_{\eps, \eps' \to 0} \int_0^\alpha \Expect{ \abs{X_\eps^\alpha - X_{\eps'}^\alpha }} d\alpha
\]
which vanishes by dominated convergence theorem. We apply this to our specific case for which we have already proven the desired pointwise convergence and for which the domination follows from \eqref{eq:proof_th_intersection1}. It implies that
\begin{align*}
& \int_0^\infty d\alpha \frac{1}{\eta} \sum_{k=0}^{K-1} \indic{\alpha \in \alpha_\eps + \frac{\eta}{\abs{\log \eps}} \left[ \frac{k}{K}, \frac{k+1}{K} \right)} \abs{\log \eps}^2 \eps^a \\
& \times \int_A dx \indic{ \frac{1}{\eps} L_{x,\eps}^{(1)} \in 2\left(\alpha - \frac{\eta}{|\log \eps|} \frac{k+1}{K}, \alpha + \frac{\eta}{|\log \eps|} \left( 1 - \frac{k}{K} \right) \right] |\log \eps|^2, \frac{1}{\eps} L_{x,\eps}^{(2)} >2 (a-\alpha - \frac{\eta}{|\log \eps|} \frac{k}{K}) |\log \eps|^2 - 2\eta |\log \eps| }
\end{align*}
converges in $L^1$ towards
\[
(1+o(1)) \int_0^a \Mc_{x_0,z_0}^{D_0,\alpha} \cap \Mc_{x_1,z_1}^{D_1,a-\alpha}(A) d\alpha.
\]
By letting $\eta \to 0$ and then $K \to \infty$, we obtain the desired upper bound:
\[
\abs{\log \eps} \eps^a \int_A dx \indic{L_{x,\eps}^{(1)} + L_{x,\eps}^{(2)} \geq 2a \eps \abs{\log \eps}^2, L_{x,\eps}^{(1)}>0, L_{x,\eps}^{(2)}>0 }
\]
is smaller or equal than a sequence which converges in $L^1$ towards 
\[
\int_0^a \Mc_{x_0,z_0}^{D_0,\alpha} \cap \Mc_{x_1,z_1}^{D_1,a-\alpha}(A) d\alpha.
\]
The lower bound is obtained along the same lines and we have shown that for all Borel set $A \subset \C$, the following convergence in $L^1$ holds:
\[
\lim_{\eps \to 0}
\abs{\log \eps} \eps^a \int_A dx \indic{L_{x,\eps}^{(1)} + L_{x,\eps}^{(2)} \geq 2a \eps \abs{\log \eps}^2, L_{x,\eps}^{(1)}>0, L_{x,\eps}^{(2)}>0 }
=
\int_0^a \Mc_{x_0,z_0}^{D_0,\alpha} \cap \Mc_{x_1,z_1}^{D_1,a-\alpha}(A) d\alpha.
\]
This concludes the proof of Propositions \ref{prop:def_measures} and \ref{prop:intersection_multipoint}.
\end{proof}

We finish with a proof of Proposition \ref{prop:other_charac}.

\begin{proof}[Proof of Proposition \ref{prop:other_charac}]
Recall that $B_{x_i,z_i}^{D_i}$ is a trajectory distributed according to $\P_{x_i,z_i}^{D_i}$ and that $L_{x,\eps}^{(i)}$ denotes the associated local times of circles $\partial D(x,\eps)$. 
Let $x \in D$. The key point of the proof is that the law of $B_{x_i,z_i}^{D_i}$ conditioned on $\{ \tfrac1{\eps} L_{x,\eps}^{(i)} \geq 2 a_i | \log \eps|^2 \}$ converges weakly to the law of $B_{x_i,x}^{D_i} \wedge \Xi_x^{D_i,a_i} \wedge B_{x,z_i}$ as $\eps \to 0$. This fact was already proven in \cite{bass1994} (see Proposition 5.1 therein). From this and the convergence of $\bigcap_{i=1}^r \Mc_{x_i,z_i;\eps}^{D_i,a_i}$ to $\bigcap_{i=1}^r \Mc_{x_i,z_i}^{D_i,a_i}$ (Proposition \ref{prop:intersection}), one easily obtains Proposition \ref{prop:other_charac}. See the proof of \cite[Proposition 6.2]{jegoGMC} for more details in the case of one single trajectory (no new input is required in the case of several trajectories).
\end{proof}

\paragraph*{Acknowledgement}

I am grateful to Nathanaël Berestycki for many inspiring discussions, to Hugo Falconet and Lucas Teyssier for comments on a first version of the article and to anonymous referees for their careful reading.

\paragraph*{Declarations -- Funding}

During the process of writing the current article, the author was recipient of a DOC Fellowship of the Austrian Academy of Sciences at the Faculty of Mathematics of the University of Vienna and was partly supported by the EPSRC grant EP/L016516/1 for the University of Cambridge Centre for Doctoral Training, the Cambridge Centre for Analysis.
The author has no relevant financial or non-financial interests to disclose.

\bibliographystyle{alpha}
\bibliography{../../bibliography}

\end{document}